\documentclass{amsart}
\usepackage[english]{babel}
\usepackage{amsfonts, amsmath, amsthm, amssymb,amscd,indentfirst}

\numberwithin{equation}{section}

\usepackage{marginnote}

\usepackage{mathrsfs}
\usepackage{euscript}

\usepackage{xcolor}
\usepackage{amsmath,amssymb,latexsym,indentfirst}
\usepackage{times}
\usepackage{palatino}
\usepackage{stmaryrd}

\usepackage{wasysym}

\usepackage{graphicx}
\graphicspath{ {./images/}}

\theoremstyle{plain}
\newtheorem{theorem}{Theorem}[section]

\newtheorem{proposition}[theorem]{Proposition}
\newtheorem{lemma}[theorem]{Lemma}
\newtheorem{corollary}[theorem]{Corollary}
\theoremstyle{definition}
\newtheorem{definition}[theorem]{Definition}

\newtheorem{example}[theorem]{Example}
\theoremstyle{remark}
\newtheorem{remark}[theorem]{Remark}

\newcommand{\interior}[1]{%
	{\kern0pt#1}^{\mathrm{o}}%
}

\begin{document}
	
	\title{Mass, center of mass and isoperimetry in asymptotically flat $3$-manifolds}
	\author{Sergio Almaraz}
	\address{Universidade Federal Fluminense (UFF), Instituto de Matem\'atica, Campus do Gragoat\'a, Rua Prof. Marcos Waldemar de Freitas, s/n, 24210-201, Niter\'oi, RJ, Brazil.
	}
	\email{sergio.m.almaraz@gmail.com}
	\author{Levi Lopes de Lima}
	\address{Universidade Federal do Cear\'a (UFC),
		Departamento de Matem\'{a}tica, Campus do Pici, Av. Humberto Monte, s/n, Bloco 914, 60455-760,
		Fortaleza, CE, Brazil.}
	\email{levi@mat.ufc.br}
	\thanks{S. Almaraz has been partially supported by  CNPq/Brazil grant 
		309007/2016-0 and FAPERJ/Brazil grant 202.802/2019, and L. de Lima has been partially supported by CNPq/Brazil grant
		312485/2018-2. Both authors have been partially supported by FUNCAP/CNPq/PRONEX grant 00068.01.00/15.}
	
	\begin{abstract}
		We revisit the interplay between the mass, the center of mass and the large scale behavior  of certain isoperimetric quotients in the setting of asymptotically flat $3$-manif\-olds (both without and with a non-compact boundary). In the boundaryless case, we first check that 
		the isoperimetric deficits involving the total mean curvature recover the ADM  mass in the asymptotic limit, thus extending a classical result due to G. Huisken. Next, under a  Schwarzschild asymptotics and assuming that the mass is positive we indicate  how the implicit function method pioneered by R. Ye  and refined by L.-H. Huang   may be adapted to establish the existence of a foliation of a neighborhood of infinity satisfying the corresponding curvature conditions. Recovering the mass as the asymptotic limit of the corresponding {relative} isoperimetric deficit also holds true in the presence of a non-compact boundary, where  we additionally obtain, again  under a Schwarzschild asymptotics, a foliation at infinity by free boundary constant mean curvature hemispheres, which are shown to be the {unique} {relative} isoperimetric surfaces for all sufficiently large enclosed volume, thus extending to this setting a celebrated result by M. Eichmair and J. Metzger. 
		Also, in each case treated here we relate the geometric center of the foliation to the center of mass of the manifold as defined by Hamiltonian methods.
	\end{abstract}

	\maketitle
	\tableofcontents

	\section{Introduction}\label{introd}

	Among the large scale invariants that can be attached by means of Hamiltonian methods to an asymptotically flat Riemannian $ 3$-manifold, viewed as the (time-symmetric) initial data set of a solution of Einstein field equations, the ADM mass and the center of mass stand out as the most relevant ones. Besides their undisputed physical prominence, the study of these  invariants also reveals  deep connections with several areas of Geometric Analysis, including the Yamabe problem \cite{schoen1984conformal,leeparker1987,brendle2015recent}, the inverse mean curvature flow \cite{huisken2001}, the construction of canonical foliations at infinity \cite{huisken1996definition,ye1997foliation,metzger2007foliations,huang2008center,huang2010foliations,nerz2018foliations} and the existence of isoperimetric surfaces for sufficiently large enclosed volumes \cite{eichmairlarge}. We recall that in order to have the center of mass well defined, we must supplement the standard ADM decay assumptions with the so-called Regge-Teitelboim conditions \cite{regge1974role,beig1987poincare}. Even though some of the results discussed below do not require the fulfillment of these extra conditions, throughout this Introduction we assume that this is always the case in order to simplify the presentation.

	Motivated by questions related to the Yamabe problem on manifold with boundary \cite{almaraz2015convergence}, a version of the positive mass theorem has been established for asymptotically flat manifolds carrying a non-compact boundary \cite{almaraz2014positive,almaraz2019spacetime}. Moreover, a notion of center of mass for this class of manifolds has been introduced in \cite{de2019mass}, again under a suitable Regge-Teitelboim-type condition. The main purpose of this paper is to confirm that, similarly to what happens in the boundaryless case, these asymptotic invariants also play a central role in the investigation of large scale isoperimetric properties in the presence of a non-compact boundary.

	With this goal in mind, we start by slightly broadening our perspective and checking  that already in the boundaryless case the ADM mass is recovered  as the asymptotic limit of certain isoperimetric deficits involving the enclosed volume, the area and the total mean curvature in various combinations (Theorem \ref{isovk}). This provides interesting extensions of a well-known result due to Huisken \cite{huisken2006isoperimetric}, where the  area/volume case is dealt with. 
	Moreover, it clearly suggests the existence of stable foliations in a neighborhood of infinity satisfying the corresponding curvature conditions, which besides the mean curvature should additionally involve (a modified version of) the Gauss-Kronecker curvature (the product of the principal curvatures). Under a Schwarzschild-type asymptotics, and assuming as usual that the ADM mass is positive, we prove the existence of the foliations by adapting the well-known implicit function method pioneered by Ye \cite{ye1997foliation} and refined by Huang \cite{huang2008center} (Theorem \ref{fol:curv:cond}). Their approach is made feasible here by means of a simple calculation expressing the Gauss-Kronecker curvature of large coordinate spheres in terms of the corresponding mean curvature  up to a remainder decaying fast enough (Proposition \ref{k:func:h}). This is a quite direct consequence of the existence of an  {\em almost} conformal vector field at infinity and allows us to rely on the computations for the mean curvature in \cite{huang2008center,huang2009center}. As a by-product of this  procedure we are able to check in each case that the geometric center of the foliation  and the center of mass of the manifold always remain at a finite distance from each other and even coincide for certain values of the parameters defining the asymptotics.
	Besides their intrinsic interest, the existence of the foliations suggest that their leaves may be realized as isoperimetric surfaces for the corresponding isoperimetric problems, which  involve minimizing the total mean curvature for large prescribed values of the volume or area, at least if some convexity assumption on the competing surfaces is imposed.  Notice that this is in alignment with the  well-known fact that round spheres in $\mathbb R^3$ constitute global minimizers to these problems among convex surfaces \cite{schneider2014convex,guan2009quermassintegral,chang2011aleksandrov}; see also Remark  \ref{iso:bm:theo} for more on this point.      
	
	Our next goal is to confirm that some of the classical results referred to above may be suitably extended to the class of asymptotically flat manifolds introduced in \cite{almaraz2014positive}. We first prove that in the presence of a non-compact boundary the {\rm relative} isoperimetric deficit involving the area of coordinate hemispheres and the volume they enclose jointly with the boundary also recovers the mass in the asymptotic limit (Theorem \ref{remwithbd}). As before, this preliminary remark is accompanied by a result ensuring, again under a Schwarzschild asymptotics and assuming that the mass is positive, that a neighborhood of infinity is foliated by stable free boundary hemispheres  of constant mean curvature (CMC) and moreover that the geometric center of mass of this foliation coincides with the center of mass as defined by Hamiltonian methods (Theorem \ref{free:af:bd}). 
	We may view these results as the large scale analogues of the main theorems in \cite{fall2009area,montenegro2019foliation}, even though the technical details are quite distinct in nature.
	A key step here is the establishment  of a certain identity relating the center of mass to the integral over large coordinate hemispheres of the higher order terms in the expansion of the corresponding mean curvature against the asymptotic coordinates along the non-compact boundary (Proposition \ref{mc:huang:mean:b}). In the boundaryless case, this kind of identity was first proved in 
	\cite{huang2008center,huang2009center} by means of a quite delicate density argument. Subsequently, an elementary proof appeared in 
	\cite{eichmair2013unique} and we succeed in checking that this latter reasoning adapts well in the presence of a non-compact boundary. As in these works, the identity 
	is used here to remove the obstruction to the invertibility of the relevant linearized operator coming from the invariance of the mean curvature under translational isometries preserving the asymptotic boundary, besides playing a crucial role in checking that the geometric and Hamiltonian centers of mass coincide. Moreover, it provides an alternate expression for the center of mass as an asymptotic integral involving the mean curvature of large coordinate hemispheres (Corollary \ref{cm:alt}). 
	Finally, we complete our analysis of the large scale geometry of this class of manifolds by checking that these free boundary CMC hemispheres constitute the {\em unique} relative isoperimetric surfaces for sufficiently large values of the enclosed volume (Theorem \ref{iso:ext:em1}). This extends to our setting a previous result by Eichmair-Metzger in the boundaryless case \cite{eichmairlarge}; see also \cite{munoz2020isoperimetric}.

	In order to keep  to a minimum the technical features of the exposition, we work here under a Schwarzschild asymptotics whenever needed. In particular, we are usually quite generous when imposing orders of decay rates for the asymptotics of geometric quantities. Also, we have chosen to avoid the consideration of inner minimal horizons (black holes). Moreover, we restrict ourselves to the time-symmetric case, thus bypassing the complications coming from the  extrinsic geometry of the given initial data set. As a consequence, we are unable to consider here the rather appealing question of determining how the asymptotic quantities vary as the initial data set evolves in time under the field equations. Nevertheless, we believe that  the results established here under these mildly restrictive assumptions may be suitably extended to larger classes of asymptotically flat manifolds, in the line of \cite{huang2010foliations,metzger2007foliations,nerz2015foliations,nerz2018foliations,eichmairlarge,chodosh2016isoperimetry,jauregui2019lower,cederbaum2018center} for instance. 
	Another topic we refrain from discussing here in detail is the uniqueness of the leaves as solutions of the corresponding variational problems. However, we remark that in the specific case of the free boundary CMC hemispheres in Theorem \ref{free:af:bd}, the appropriate uniqueness is easily obtained by adapting an argument in \cite[Section 4]{huisken1996definition}; see Appendix \ref{uniq:stab}. We note that this uniqueness is crucial when identifying those hemispheres to the relative isoperimetric hemispheres in Theorem \ref{iso:ext:em1} and it eventually guarantees that they are unique in the class of relative isoperimetric surfaces enclosing large volumes.
	Finally, it would be interesting to investigate how our constructions fit into the quasi-local approach to conserved quantities summarized in \cite{Chen:2013lza}. We hope to address some of these questions elsewhere.

	This paper is organized as follows. In section \ref{sec:statem}, after a brief motivation intended to illustrate the local interplay between isoperimetric quotients and the scalar curvature, we provide precise statements of all the results mentioned above. The arguments leading to the existence of the stable foliations are presented in Sections \ref{const:fol} and \ref{stab:fol}. The proof of Theorem \ref{iso:ext:em1}, which provides the relative isoperimetric regions in the presence of a boundary, is explained in Section \ref{large:rel:iso}. In order not to interrupt the exposition in the bulk of the paper, we  defer to the appendices the proofs of a few technical results, including a discussion of the isoperimetric variational theories involved, specially in regard to the corresponding stability criteria. These appendices  are also used to introduce much of the notation used in the paper. Finally, we also intersperse along the text a few interesting problems in this area of research. 
	
	\vspace{0.3cm}
	\noindent
	{\bf Acknowledgments.} We thank A. Freitas, E. Lima, J.F. Montenegro and S. Nardulli for conversations at an early stage of this project.

	\section{Preliminaries and statements of the results}\label{sec:statem}
	
	For the sake of motivation, we start by considering an arbitrary Riemmanian $3$-manifold $(M,g)$. We fix $q\in M$ and introduce normal coordinates $z=(z_1,z_2,z_3)$ around $q$ in the usual way. If 
	$\mathsf r=|z|$ 
	is the geodesic distance to $q$, let $B_{\mathsf r}^M(q)$ be the geodesic ball of radius ${\mathsf r}$ centered at $q$ and $S_{\mathsf r}^M(q)=\partial B_{\mathsf r}^M(q)$ the corresponding geodesic sphere. If $N$ is the outward unit normal to $S_{\mathsf r}^M(q)$,
	let $W_{\mathsf r}=\nabla N$ be the shape operator of $S_{\mathsf r}(q)$, where $\nabla$ is the Levi-Civita connection. For each $k=1, 2$ we may consider the curvature integral
	\begin{equation}\label{}
		Q^{M;k}_{\mathsf r}(q)=\int_{S^M_{{\mathsf r}}(q)}\sigma_{2-k}(W_{\mathsf r}) dS^M_{\mathsf r}(q),
	\end{equation}
	where $\sigma_i(W_{\mathsf r})$ is the elementary symmetric function of degree $i
	$ in the eingenvalues of $W_{\mathsf r}$ (the principal curvatures $\kappa_1,\kappa_2$). Thus, $\sigma_0=1$, $\sigma_1=\kappa_1+\kappa_2$, the mean curvature, also denoted here by $H$, and $\sigma_2=\kappa_1\kappa_2$, the Gauss-Kronecker curvature, also denoted here by $K$. 
	Finally, we set $Q_{\mathsf r}^{M;3}(q)={\rm vol}_g(B_{\mathsf r}^M(q))$ by convention.
	
	For $1\leq k<h\leq 3$ we consider the isoperimetric quotient
	\begin{equation}\label{isoper:k,h}
		I^{M;h,k}_{\mathsf r}(q)=\frac{Q^{M;k}_{\mathsf r}(q)^{\frac{h}{k}}}{Q^{M;h}_{\mathsf r}(q)}.
	\end{equation}
	Notice that for $(M,g)=(\mathbb R^3,\delta)$, the Euclidean space with the standard flat metric, these quotients do not depend on the pair $(q,{\mathsf r})$ so we denote them simply by $I^{h,k}$. 
	
	\begin{proposition}\label{asymp:isop:local}
		As ${\mathsf r}\to 0$, there exists $c_{k,h}>0$ such that
		\begin{equation}\label{iso:arb:man}
			1-\frac{I^{M;h,k}_{\mathsf r}(q)}{I^{h,k}}=c_{k,h}\mathsf R_g(q){\mathsf r}^2+O({\mathsf r}^4),
		\end{equation}
		where $\mathsf R_g$ is the scalar curvature of $g$. 
		In particular, if $\mathsf R_g(q)\geq 0$ then 
		\begin{equation}\label{subiso:q}
			I^{M;h,k}_{\mathsf r}(q)\leq I^{h,k},
		\end{equation}
		for all ${\mathsf r}>0$ small enough. 
	\end{proposition}
	
	\begin{proof}
		This folklore result may be checked as follows. The expansion for  the volume, namely, 
		\[
		\frac{Q_{\mathsf r}^{M;3}(q)}{Q_{\mathsf r}^{\mathbb R^3;3}(\vec 0)}=1-\frac{\mathsf R_g(q)}{30}{\mathsf r}^2+O({\mathsf r}^4),
		\] 
		may be found in \cite[Section 9.2]{gray2012tubes}.
		By means of the well-known variational formulae 
		\[
		Q_{\mathsf r}^{M;2}(q)=\frac{d}{d{\mathsf r}}Q_\rho^{M;3}(q),\quad Q_{\mathsf r}^{M;1}(q)=\frac{d}{d{\mathsf r}}Q_\rho^{M;2}(q),
		\]
		we see that 
		\[
		\frac{Q_{\mathsf r}^{M;2}(q)}{Q_{\mathsf r}^{\mathbb R^3;2}(\vec 0)}=1-\frac{\mathsf R_g(q)}{18}{\mathsf r}^2+O({\mathsf r}^4),
		\]
		and 
		\[
		\frac{Q_{\mathsf r}^{M;1}(q)}{Q_{\mathsf r}^{\mathbb R^3;1}(\vec 0)}=1-\frac{\mathsf R_g(q)}{9}{\mathsf r}^2+O({\mathsf r}^4).
		\]
		From these, the expansions in (\ref{iso:arb:man}) follow  easily. 
	\end{proof}
	
	This result says that the local behavior of the relative isoperimetric quotients ${I^{M;h,k}_{\mathsf r}(q)}/{I^{h,k}}$ is completely determined by the sign of $\mathsf R_g(q)$.
	We note however that in general this isoperimetric comparison result only holds true at small scales since the argument depends on the corresponding asymptotic expansions as ${\mathsf r}\to 0$. Our first remark here is that large scale versions of this principle hold true in the setting of asymptotically flat manifolds with non-negative scalar curvature and, as we shall see, this is closely related to the positive mass theorem in General Relativity. In the case $(h,k)=(3,2)$, this was first observed by Huisken \cite{huisken2006isoperimetric}, as we now pass to describe. 
	
	\begin{definition}\label{af:eman}
		A manifold $(M,g)$ is asymptotically flat  if there exists a compact subset $U\subset M$ and a diffeomorhism $M\backslash U\cong \mathbb R^3\backslash B_1(\vec 0)$ such that in the corresponding asymptotic coordinates $x=(x_1,x_2,x_3)$ there holds   
		\begin{equation}\label{exp:met:af:bd} 
			e:=g-\delta =O_2(r^{-\tau}), \quad \tau>\frac{1}{2}
		\end{equation}
		and 
		\begin{equation}\label{scalar:asym:bd}
			\mathsf R_g=O(r^{-3-\sigma}), \quad \sigma>0,
		\end{equation}
		as $r=|x|_\delta\to +\infty$.  
	\end{definition}
	
	Note that (\ref{scalar:asym:bd}) implies $\mathsf R_g\in L^1(M)$. Also, for a tensor $f=f(x)$ in the asympotic region, we say that $f=O_k(r^{-\tau})$ if $|(\partial_\alpha f)(x)|=O(r^{\tau-|\alpha|})$ for any multi-index $\alpha$ with  $0\leq |\alpha|\leq k$.
	
	Any manifold as in Definition \ref{af:eman} may be viewed as the time-symmetric initial data set of a solution to the Einstein field equations whose geometry at spatial infinity is essentially Minkowsk\-ian. In particular, the appropriate Hamiltonian version of Noether's theorem may be employed to attach to the Riemannian manifold $(M,g)$ certain asymptotic invariants capturing common physical quantities associated to the  isolated gravitational system modeled by the solution \cite{arnowitt1962gravitation, regge1974role, beig1987poincare,christodoulou2008mathematical, harlow2020covariant, delima2021conserved}. The most prominent of these invariants is the so-called ADM mass, which is given  by
	\begin{equation}\label{adm:mass}
		m_{ADM}=\lim_{r\to +\infty}\frac{1}{16\pi}\int_{S^2_r}{\mathbb U}({\bf 1},e)\left(\frac{x}{r}\right)dS_r^2.
	\end{equation}
	Here,  
	\[
	{\mathbb U}(f,e)=f({\rm div}_\delta e-d{\rm tr}_\delta e)-{\bf i}_{\nabla_\delta f} e+{\rm tr}_\delta e\, df, 
	\]
	$f:\mathbb R^3\to\mathbb R$ is a smooth function, ${\bf 1}$ is the function identically equal to $1$, and 
	$S_r^2$ is the coordinate sphere of radius $r$ in the asymptotic region centered at the origin. 
	
	Let us  denote by $A(r)$ (respectively $V(r)$) the area of $S_r^2$ (respectively the volume of the compact region enclosed by $S^2_r$). 
	We now recall a classical definition due to Huisken.
	
	\begin{definition}\label{huisiso}\cite{huisken2006isoperimetric}
		Under the conditions above, we set  
		\begin{equation}\label{huiiso2}
			J_r^{M;3,2}=\frac{2}{A(r)}\left(V(r)-\frac{1}{6\pi^{1/2}}A(r)^{\frac{3}{2}}\right).
		\end{equation}
	\end{definition}
	
	\begin{remark}\label{remhuis}
		This may be expressed as 
		\begin{equation}\label{huiiso3}
			J^{M;3,2}_r=2\frac{V(r)}{A(r)}\left(1-{\frac{\widehat I^{M;3,2}_r}{I^{3,2}}}\right),
		\end{equation}
		where
		\[
		\widehat I^{M;3,2}_r=\frac{A(r)^{3/2}}{V(r)}
		\]
		is the  large scale analogue of the isoperimetric quotient in (\ref{isoper:k,h}) with $(h,k)=(3,2)$ and 
		$I^{3,2}=6\pi^{1/2}$ is the corresponding quotient  evaluated on round spheres in $(\mathbb R^3,\delta)$. Thus, $J^{M;3,2}_r$ should be thought of as the asymptotic analogue of the isoperimetric deficit in  (\ref{isodef:(h,k)}) with $(h,k)=(3,2)$. 
	\end{remark}
	
	As remarked above,  the isoperimetric deficit $ J^{M;3,2}_r$ should somehow be controlled in case the standard dominant energy condition $\mathsf R_g\geq 0$ holds; here we view $(M,g)$ as a time-symmetric initial data set propagating in time to a solution of Einstein field equations, so this energy condition is justified on physical grounds. This link between scalar curvature and isoperimetry holds true indeed and the key result goes as follows.
	
	\begin{theorem}\label{huisfst} \cite{huisken2006isoperimetric} One has
		\begin{equation}\label{huisfst:2}
			\lim_{r\to +\infty}J_r^{M;3,2}=m_{ADM}.
		\end{equation}
	\end{theorem}
	
	Notice that this justifies the appearance of the extra factor $V(r)/A(r)$  in the right-hand side of (\ref{remhuis}) as the ADM mass has the dimension of length.  
	
	Combining this with the  positive mass theorem \cite{schoen1979proof} we obtain the following remarkable result.
	
	\begin{corollary}\label{remarkisolarge}
		If $(M,g)$ is asymptotically flat with scalar curvature $\mathsf R_g\geq 0$ then for $r>0$ large enough,
		\[
		{\widehat I^{M;3,2}_r}\leq {I^{3,2}},
		\] 
		with the strict inequality holding unless $(M,g)=(\mathbb R^3,\delta)$ isometrically.
	\end{corollary}
	
	\begin{proof}
		We know that $m_{ADM}\geq 0$ by the positive mass theorem. If $m_{ADM}=0$ then $(M,g)=(\mathbb R^3,\delta)$ isometrically and the equality holds. Otherwise, $m_{ADM}>0$ and the strict inequality holds. 
	\end{proof}
	
	\begin{remark}\label{sub:isop:rem}
		The discussion above leads naturally to a conjectured $C^0$ version of the positive mass theorem. More precisely, the subisoperimetry condition in (\ref{subiso:q}) with $(h,k)=(3,2)$ may be interpreted  as the statement that $\mathsf R_g(q)\geq 0$ for a metric $g$ which is only assumed to be $C^0$. 
		This led  Huisken to conjecture that the validity of this subisoperimetry condition at any $q\in M$ implies that the mass of an asymptotically flat manifold $(M,g)$ is nonnegative, where $g$ is of class $C^0$, with the equality holding if and only if $(M,g)=(\mathbb R^3,\delta)$ isometrically. Of course, here  the mass of $(M,g)$ is defined by the limit in the left-hand side of (\ref{huisfst:2}), whenever it exists. To our knowledge, this conjecture, which constitutes a far-reaching generalization of the standard positive mass theorem in \cite{schoen1979proof},  is still wide open. In any case, we observe that a similar question may be formulated in the presence of a non-compact boundary; see Remark \ref{subiso:bd} below. 
	\end{remark}

	A proof of Theorem \ref{huisfst} appears in \cite{fan2007large} and 
	for our purposes a key observation is that a simple variation of their  computations, which we reproduce in Appendix \ref{isop:mass:app} below, proves that this remarkable connection between isoperimetry and mass also holds true for the other isoperimetric quotients involving  the total mean curvature
	\[
	M(r):=\int_{S_r^2}HdS^2_r
	\]
	of the coordinate sphere $S_r^2$. They correspond to the cases $(h,k)=(3,1)$ and $(h,k)=(2,1)$ in (\ref{iso:arb:man}) and are defined as follows. 
	
	\begin{definition}\label{isodef:(h,k)}
		Under the conditions above,  set 
		\[
		J^{M;3,1}_r=\frac{4}{3r M(r)}\left( V(r)-\frac{1}{3\cdot 2^7\cdot\pi^2} M(r)^3\right)=\frac{4}{3}\frac{ V(r)}{r M(r)}\left(1- \frac{\widehat{I}_r^{M;3,1}}{I^{3,1}}\right), 
		\] 
		and 
		\[
		J^{M;2,1}_r=\frac{1}{M(r)}\left(A(r)-\frac{1}{16\pi}M(r)^2\right)=\frac{A(r)}{M(r)}\left(1- \frac{\widehat{I}_{r}^{M;2,1}}{I^{2,1}}\right),
		\]
		where 
		\[
		\widehat{I}_{r}^{M;3,1}=\frac{M(r)^3}{V(r)},\quad \widehat{I}_{r}^{M;2,1}=\frac{M(r)^2}{A(r)}, \quad I^{3,1}=3\cdot 2^7\cdot \pi^2,\quad I^{2,1}=16\pi.
		\]
	\end{definition}

	\begin{theorem}\label{isovk}
		If $(M,g)$ is asymptotically flat then
		\begin{equation}\label{hkiso2}
			\lim_{r\to +\infty} J^{M;3,1}_r=m_{ADM},\quad \lim_{r\to +\infty} J^{M;2,1}_r=m_{ADM}.
		\end{equation}
	\end{theorem}

	Again, if combined with the positive mass theorem, this result has the following nice consequence. 
	
	\begin{corollary}\label{remarkisolargeh}
		If $(M,g)$ is asymptotically flat with scalar curvature $\mathsf R_g\geq 0$ then for all $r>0$ large enough, we have the inequalities
		\[
		{\widehat I^{M;3,1}_r}\leq{I^{3,1}}, \quad 	{\widehat I^{M;2,1}_r}\leq{I^{2,1}},
		\] 
		with the strict inequality holding in each case unless $(M,g)=(\mathbb R^3,\delta)$ isometrically.
	\end{corollary}

	We present the proof of Theorem \ref{isovk} in Appendix \ref{isop:mass:app}. 
	
	\begin{remark}\label{sub:isop:c1}
		The subisoperimetry condition in (\ref{subiso:q}) with either $(h,k)=(3,1)$ or  $(h,k)=(2,1)$ may be interpreted  as the statement that $\mathsf R_g(q)\geq 0$ for a metric $g$ which is only assumed to be $C^1$.  Thus, similarly to the discussion in Remark \ref{sub:isop:rem}, we may ask whether the validity of any of these subisoperimetry conditions everywhere on an asymptotically flat manifold endowed with a $C^1$ metric implies that the associated mass in nonnegative, with the rigidity statement holding as well. Here, the mass should be defined by the corresponding limit in the left-hand sides of (\ref{hkiso2}), whenever it exists. It would be interesting to examine this question in light of recent developments on the positive mass theorem for metrics with low regularity; see  \cite{lee2015positive} and the references therein. 
	\end{remark}
	
	We now recall the connection between isoperimetry and another basic invariant of asymptotically flat $3$-manifolds, namely, the center of mass \cite{regge1974role,beig1987poincare}. In the following, given a tensor $f=f(x)$ on the asymptotic region, we set $f^{(1)}(x)=(f(x)+f(-x))/2$ and $f^{(-1) }=f-f^{(1)}$, so that $f^{(1)}$ and $f^{(-1)}$ are the even and odd parts of $f$, respectively. Also, we say that $f^{(1)}=\widehat O_k(r^{-\tau})$ (respectively, $f^{(-1)}=\widehat O_k(r^{-\tau})$) if 
	$(\partial_\alpha f)^{((-1)^{|\alpha|})}(x)=O(r^{-\tau-|\alpha|})$
	(respectively, $(\partial_\alpha f)^{((-1)^{1+|\alpha|})}(x)=O(r^{-\tau-|\alpha|})$) for any multi-index $\alpha$ with $0\leq|\alpha|\leq k$.

	\begin{definition}\label{regge:cond}
		We say that an asymptotically flat $3$-manifold as in Definition \ref{af:eman} satisfies the Regge-Teitelboim (RT) conditions if there holds
		\begin{equation}\label{regge:met}
			g^{(-1)}(x)={\widehat O}_2(r^{-1-\tau}), \quad \tau>\frac{1}{2},
		\end{equation}
		and 
		\begin{equation}\label{regge:scal}
			\mathsf R_g^{(-1)}(x)={\widehat O}(r^{-4-\sigma}), \quad \sigma>0.
		\end{equation}	
	\end{definition}
	
	Note that (\ref{regge:scal}) guarantees that each $x_i\mathsf R^{(-1)}_g\in L^1 (M)$ for $i=1,2,3$. Also, it is clear that $x_i\mathsf R^{(1)}_g$
	has the property that its integral over the region enclosed by two coordinate spheres vanishes. So, we may use the method in \cite{michel2011geometric}, with $x_i$ as a static potential, to ensure that  
	to any manifold as in Definition \ref{regge:cond} with $m_{ADM}\neq 0$ we may attach a (Hamiltonian) {center of mass} defined by 
	\begin{equation}\label{centmass}
		\mathcal C_i=\lim_{r\to +\infty}\frac{1}{16\pi m_{ADM}}\int_{S_{r}^2}\mathbb U(x_i,e)\left(\frac{x}{r}\right)dS_{r}^2,\quad i=1,2,3.
	\end{equation}

	\begin{example}\label{ex:schwarz}
		Recall that the Schwarzschild metric on $\mathbb R^3\backslash\{c\}$, $c\in\mathbb R^3$, is given by 
		\begin{equation}\label{sch:met}
			g_{m,c}(x)=\left(1+\frac{m}{2|x-c|_\delta}\right)^4\delta,
		\end{equation}
		where $m\in\mathbb R$. We denote $g_m=g_{m,\vec 0}$ for simplicity. 
		Clearly, $g_{m,c}$ is asymptotically flat and since $\mathsf R_{g_{m,c}}=0$, it also satisfies the RT conditions. The associated asymptotic invariants may be easily determined if we expand (\ref{sch:met}) as $r=|x|_\delta\to+\infty$:
		\begin{equation}\label{exp:flex}
			{{g_{m,c}(x)}}=\left(1+\frac{2m}{r}+\frac{2m c\cdot x }{r^3}+\frac{3m^2}{2r^2}+O(r^{-3})\right)\delta,
		\end{equation}
		where the dot is the Euclidean inner product. 
		A direct computation then gives  $m_{ADM}=m$ and ${{\mathcal C=c}}$.
	\end{example}
	
	This example indicates that the mass and the center of mass are somehow captured {{by, respectively, the first order term and the odd part of the second order term in the asymptotic expansion of the metric in $r^{-1}$}}. This motivates the consideration of the following important class of examples.

	\begin{definition}\label{asymp:schwarz}
		We say that an asymptotically flat  $3$-manifold $(M,g)$ (with $\tau=1$) is {\em{asymp\-tot\-ically Schwarzschild}} (aS) if in the asymptotic region there holds  
		\begin{equation}\label{asym-sch:def}
			g-\left(1+\frac{2m}{r}\right)\delta=O_3(r^{-2}).
		\end{equation}
	\end{definition}
	
	In general, an aS manifold may not satisfy the RT conditions as (\ref{regge:scal}) might be violated. 
	In this case, we would still have $m_{ADM}=m$ but if $m\neq 0$ the limit in (\ref{centmass}) might not exist. In fact, convergence takes place  if and only if each $x_i\mathsf R_g$ is integrable \cite{chan2014note}. 
	Examples where the limit defining the center of mass diverges may be found in \cite{cederbaum2015explicit}.
	
	\begin{definition}\label{def:asrtman}		
		If an aS manifold satisfies the RT conditions we say that it is an {\it{aSRT manifold}}.
	\end{definition}
	
	Motivated by (\ref{exp:flex}), we also consider a more flexible kind of asymptotics.

	\begin{definition}\label{def:e-as}	
		We say that an asymptotically flat $3$-manifold $(M,g)$ (with $\tau=1$) is {\em{$\epsilon$-asymp\-tot\-ically Schwarzschild}} ($\epsilon$-aS) if in the asymptotic region there holds  
		\begin{equation}\label{asym-req}
			g-f_{m,c}^{\gamma_1,\gamma_2}\delta=O_3(r^{-2-\epsilon}),\qquad f_{m,c}^{\gamma_1,\gamma_2}=1+\frac{2m}{r}+\gamma_1\frac{c\cdot x}{r^3}+\gamma_2\frac{1}{r^2}
		\end{equation}
		with $\epsilon \geq 0$, $m\neq 0$, $\gamma_1, \gamma_2\in\mathbb R$ and $c\in\mathbb R^3$. 
	\end{definition}

	\begin{remark}\label{asym:req}
		Clearly, an $\epsilon$-aS manifold is aS. Also, the case $\epsilon=0$ does not impose any further conditions on an aS $3$-manifold  and $\gamma_1$, $\gamma_2$ and $c$ are irrelevant in this case.
		If $\epsilon>0$ then $(M,g)$ is aSRT with center of mass 
		$$\mathcal C=\frac{\gamma_1}{2m}c.$$ 
		The most relevant choice in this case is $\gamma_1=2m$ and $\gamma_2=3m^2/2$, when satisfying \eqref{asym-req} is equivalent to 
		$$
		g-g_{m,c}=O_3(r^{-2-\epsilon}).
		$$
	\end{remark}

	With the appropriate definitions at hand, we now come back to the isoperimetry discussion motivated by Theorems \ref{huisfst} and \ref{isovk}. These results suggest that for aSRT manifolds with $m> 0$, large coordinate spheres  might be perturbed to yield global solutions to the corresponding isoperimetric problems (here we should assume that the competing surfaces are convex in a suitable sense so as to make sure that $M(r)> 0$).  In the classical case $(h,k)=(3,2)$,
	this has been established in \cite{eichmairlarge}. Very likely, a similar result should hold in the remaining cases treated in Theorem \ref{isovk} and we hope to address this issue elsewhere. 
	
	In any case, a first check toward this goal is to show that under these same asymptotic conditions a neighborhood of infinity can be foliated by surfaces satisfying the corresponding curvature conditions and moreover that the geometric center of the foliations relate to the center of mass of $(M,g)$ defined by Hamiltonian methods as in (\ref{centmass}). In the case $(h,k)=(3,2)$, the existence of a canonical foliation by stable CMC spheres and the identification of the geometric center of this foliation with the Hamiltonian center of mass has been first established by Huisken and Yau \cite{huisken1996definition} and then investigated further by many authors under varying asymptotic assumptions; see for instance  \cite{ye1997foliation, qing2007uniqueness,corvino2008center,metzger2007foliations,huang2008center,huang2010foliations, eichmair2013unique,chodosh2016isoperimetry,nerz2015foliations,nerz2018foliations,cederbaum2018center}. 
	Here we rely on the implicit function approach developed in \cite{ye1997foliation, huang2008center,huang2010foliations,huang2009center} to treat  
	the cases $(h,k)=(3,1)$ and $(h,k)=(2,1)$ as follows.

	\begin{theorem}\label{fol:curv:cond}
		Let $(M,g)$ be an aSRT $3$-manifold with positive ADM mass. Then there exists a neighborhood of infinity which is foliated by {strictly stable} spheres satisfying any of the curvature conditions
		\begin{equation}\label{curcond}
			\widetilde K:=K-\frac{1}{2}{\rm Ric}_g(\nu,\nu)={\rm const.},
		\end{equation}
		or 
		\begin{equation}\label{curcond2}
			\widetilde K=\gamma H.
		\end{equation}
		Here, $K$ is the Gauss-Kronecker curvature of the leaf, $\nu$ is its outward unit normal and $\gamma$ is a suitably chosen Lagrangian multiplier (varying with the leaf). Also, the geometric center $\mathcal C_{\mathcal F}$ of any of these foliations remains at a finite distance from the Hamiltonian center of mass $\mathcal C$ of $(M,g)$. 	
		Finally, if we further assume that the manifold is $\epsilon$-aS with $\epsilon>0$ then $\mathcal C_{\mathcal F}=\mathcal C$. 
	\end{theorem}
	
	The curvature conditions  (\ref{curcond}) and (\ref{curcond2}) express the fact that the surfaces are critical configurations for the corresponding isoperimetric quotients, i.e., they extremize the total mean curvature under a volume or area constraint, respectively. The stability statement should be understood in this variational sense; see Appendix \ref{var:setup} for a discussion on this issue. 
	The proof of Theorem \ref{fol:curv:cond} is presented in Section \ref{const:fol}. 
	\begin{remark}\label{tech:prob}
		It is not clear whether the rather unsatisfactory assumption $\epsilon>0$ in the last statement of Theorem \ref{fol:curv:cond} is an artifact of  our method of proof or an intrinsic feature of the invariants involved.   
	\end{remark}

	\begin{remark}\label{iso:bm:theo}
		It is well-known that in Euclidean space $(\mathbb R^3,\delta)$ the isoperimetric quotients in (\ref{isodef:(h,k)}) with $(h,k)=(3,1)$ or $(h,k)=(2,1)$ are minimized at round spheres if some convexity assumption is imposed on the competing surfaces. In the convex case ($K\geq 0$) this follows from classical Brunn-Minkowski theory \cite{schneider2014convex}. The result holds more generally if the competing surfaces are assumed to be mean convex ($H\geq 0$) and star-shaped \cite{guan2009quermassintegral} and it is conjectured that star-shapedness  may be dispensed with; see  \cite{chang2011aleksandrov} for a survey of recent contributions in this direction.  
		The problem of extending these isoperimetric inequalities (and their analogues involving quermassintegrals of higher order) to other geometries, again under suitable convexity requirements, is quite challenging from a technical viewpoint. 
		The class of aSRT $3$-manifolds seems to be a natural choice for starting this research program, as it is reasonable to conjecture that the leaves of the foliations in Theorem \ref{fol:curv:cond} constitute {\em global} solutions to the corresponding isoperimetric problems for large values of the area or volume. See \cite{lazaro2003cauchy} for a possible approach to a version of this problem in space forms using methods from Integral Geometry.  
	\end{remark}

	We now turn {{our attention}} to asymptotically flat $3$-manifolds carrying a non-compact boundary as in \cite{almaraz2014positive}. We set $\mathbb R^{3}_+=\{x\in\mathbb R^3; x_3\geq 0\}$, the Euclidean half-space with the standard flat metric $\delta^+=\delta|_{\mathbb R^3_+}$.

	\begin{definition}\label{af:eman:bd}(\cite{almaraz2014positive})
		A $3$-manifold $(M,g)$ is asymptotically flat with a non-compact boundary $\Sigma$ if there exists a compact subset $V\subset M$ and  a diffeomorhism  $M\backslash V\cong\mathbb R^3_{+}\backslash B_1(\vec 0)$ such that in the corresponding asymptotic coordinates $x=(x_1,x_2,x_3)$ there holds   
		\begin{equation}\label{exp:met:af}
			e^+:=g-\delta_+ =O_3(r^{-\tau}), \quad \tau>\frac{1}{2},
		\end{equation}
		and 
		\begin{equation}\label{scalar:asym}
			\mathsf R_g=O(r^{-3-\sigma}), \quad H_\Sigma=O(r^{-2-\sigma}), \quad \sigma>0,
		\end{equation}
		as $r=|x|_\delta\to +\infty$. Here,  $H_\Sigma$ is the mean curvature of $\Sigma$.  
	\end{definition}
	
	Note that (\ref{scalar:asym}) implies that $\mathsf R_g\in L^1(M)$ and $H_\Sigma\in L^1(\Sigma)$. 
	
	In this setting, the asymptotic invariant corresponding to (\ref{adm:mass}) has been defined in \cite{almaraz2014positive} and is given by 
	\begin{equation}\label{mass:form:bd}
		\mathfrak m=\lim_{r\to +\infty}\frac{1}{16\pi}\left(\int_{S^2_{r,+}}\mathbb U({\bf 1},e^+)\left(\frac{x}{r}\right)dS^2_{r,+}-\int_{S^1_r}e^+\left(\frac{x}{r},\vartheta\right)dS^1_r\right),
	\end{equation}
	where 
	$S^2_{r,+}$ is the coordinate hemisphere of radius $r$, centered at the origin of $\mathbb R^3$, in the asymptotic region, 
	$S^1_r=\partial S^2_{r,+}\subset\Sigma$ and $\vartheta$ is the outward unit {{co-}}normal vector field to $S^2_{r,+}$ along $S^1_r$ (with respect to the flat metric $\delta^+$).
	In order to introduce the corresponding {\em relative} isoperimetric deficit
	we denote by $\mathcal A(r)$ (respectively $\mathcal V(r)$) the area of $S^2_{r,+}$ (respectively, the volume of the compact region enclosed by $S^2_{r,+}$ and  $\Sigma$). Similarly to (\ref{huiiso2}) we may consider
	\[
	\mathcal J_r^{M;3,2}=\frac{1}{\mathcal A(r)}\left(\mathcal V(r)-\frac{1}{3\cdot 2^{1/2}\pi^{1/2}}\mathcal A(r)^{\frac{3}{2}}\right)=\frac{\mathcal V(r)}{\mathcal A(r)}\left(1-{\frac{{\mathcal I}^{M;3,2}_r}{\mathcal I^{3,2}}}\right),
	\]
	where 
	\[
	{\mathcal I}^{M;3,2}_r=\frac{\mathcal A(r)^{\frac{3}{2}}}{\mathcal V(r)}
	\]
	and $\mathcal I^{3,2}=3\cdot 2^{1/2}\pi^{1/2}$ is the relative isoperimetric quotient {{of}} a hemisphere centered at a point in $\mathbb R^2=\partial \mathbb R^3_+$. Similarly to Theorem \ref{huisfst} we now have 
	
	\begin{theorem}\label{remwithbd}
		Under the conditions above, 
		\begin{equation}\label{remwithbd2}
			\lim_{r\to+\infty}\mathcal J_r^{M;3,2}=\mathfrak m.
		\end{equation}
	\end{theorem}
	
	This result, whose proof is postponed to Appendix \ref{isop:mass:app}, has the following notable consequence.
	
	\begin{theorem}\label{remarkisolargebd}
		Let  $(M,g)$ be asymptotically flat with  a non-compact boundary $\Sigma$ as above. Assume further that  $\mathsf R_g\geq 0$ everywhere and $H_\Sigma\geq 0$ along the boundary. Then for all $r>0$ large enough,
		\[
		{{\mathcal I}^{M;3,2}_r}\leq{\mathcal I^{3,2}},
		\] 
		with the strict inequality holding unless $(M,g)=(\mathbb R^3_+,\delta^+)$ isometrically.
	\end{theorem}
	\begin{proof}
		Apply the positive mass theorem in \cite{almaraz2014positive}.
	\end{proof}
	
	\begin{remark}\label{subiso:bd}
		{Pick {\em any} smooth Riemannian $3$-manifold $(M,g)$ with boundary $\partial M$ and fix $q\in\partial M$. After introducing Fermi coordinates around $q$, we may consider the coordinate hemisphere $S_{\mathsf r,+ }^{M}(q)$ of small radius $\mathsf r>0$  centered at $q$. By using the calculations in \cite{fall2009area,montenegro2019foliation} we may easily check that, as $\mathsf r\to 0$,   
			\[
			1-\frac{\mathcal I^{M,\partial M;3,2}_{\mathsf r}(q)}{\mathcal I^{3,2}}=c^+H_{\partial M}(q)\mathsf r+O(\mathsf r^2).
			\] 
			Here, $c^+>0$ is a universal constant and
			\[
			\mathcal I^{M,\partial M;3,2}_{\mathsf r}(q)=\frac{{\mathcal A(\mathsf r)}^{3/2}}{\mathcal V(\mathsf r)}, 
			\]
			where  $\mathcal A(\mathsf r)$ and $\mathcal V(\mathsf r)$ are respectively the area of $S_{\mathsf r,+}^{M,\partial M}(q)$ and the volume it encloses jointly with $\partial M$, and $H_{\partial M}(q)$ is the mean curvature of $\partial M$ at $q$. Thus, if the metric $g$ is assumed to be merely $C^0$, we may interpret the (boundary) subisoperimetry condition
			\begin{equation}\label{sub:iso:bd}
				\mathcal I^{M,\partial M;3,2}_{\mathsf r}(q)\leq \mathcal I^{3,2},
			\end{equation}
			for all $\mathsf r>0$ small enough,
			as the statement that $H_{\partial M}(q)\geq 0$ in this  weak sense. Thus, in the spirit of Remark \ref{sub:isop:rem}} above, we are led to conjecture that for an asymptotically flat $3$-manifold with a non-compact boundary endowed with a $C^0$ metric, the fulfillments of the subisoperimetry conditions (\ref{subiso:q}) for $q$ in the interior and (\ref{sub:iso:bd}) for $q$ on the boundary imply that the mass of the manifold is nonnegative, with the equality taking place if only if it is isometric to  $(\mathbb R^3_+,\delta^+)$. Of course, here the mass should be defined by the limit in the left-hand side of (\ref{remwithbd2}), whenever it exists. Presumably, after a doubling across the boundary, this turns out to be equivalent to the original conjecture in Remark \ref{sub:isop:rem}. In any case, this would lead to a far-reaching generalization of the main result in \cite{almaraz2014positive}. 
	\end{remark}

	In order to have a well defined center of mass in the setting of Definition \ref{af:eman:bd}, we need the analogue of the RT conditions (\ref{regge:met}) and (\ref{regge:scal}). 
	Here, given a tensor $f=f(x)$ in the asymptotic region, we set $f^{(1')}(x)=(f(x',x_3)+f(-x',x_3))/2$, with $x'=(x_1,x_2)$, and $f^{(-1') }=f-f^{(1')}$.
	Since, in this region, $x_3$  is an {\em even} function of $x'$ which vanishes on $\Sigma$, namely, $x_3=\sqrt{r^2-|x'|^2}$, we may think of $f^{(1')}$ and $f^{(-1')}$ as the even and odd parts of $f$. In particular, we can mimic the discussion preceding Definition \ref{regge:cond} so as to make sense of any of the conditions $f^{(1')},f^{(-1')}\in \widehat O_k(r^{-\tau})$.

	\begin{definition}\label{regge:af:bd}
		We say that an asymptotically flat $3$-manifold with a non-compact boundary as in Definition \ref{af:eman:bd} satisfies the Regge-Teitelboim (RT) conditions if there holds
		\begin{equation}\label{regge:met:bd}
			g^{({-1}')}(x)={\widehat O}_2(r^{-1-\tau}), \quad \tau>1/2,
		\end{equation}
		and 
		\begin{equation}\label{regge:scal:bd}
			\mathsf R_g^{({-1}')}(x)={\widehat O}(r^{-4-\sigma}), \quad H_\Sigma^{({-1}')}={\widehat O}(r^{-3-\sigma}), \quad\sigma>0.
		\end{equation}		
	\end{definition}
	Note that (\ref{regge:scal:bd}) implies that each $x_\alpha\mathsf R_g^{({-1}')}\in L^1(M)$ and each $x_\alpha H_\Sigma^{({-1}')}\in L^1(\Sigma)$, $\alpha=1,2$. Also, $x_\alpha \mathsf R_g^{({1}')}$
	%, where $\mathsf R_g^{({\rm even}')}=\mathsf R_g-\mathsf R_g^{({\rm even}')}$, 
	has the property that its integral over the region enclosed by two large coordinate hemispheres  vanishes, and similarly for the integral of $x_\alpha H_\Sigma^{({1}')}$ over the region in the boundary enclosed by two coordinate circles. Thus, we may use the method in \cite{michel2011geometric}, as adapted in \cite{almaraz2019spacetime} and with $x_\alpha$ as a static potential, to ensure that 
	to any manifold as in Definition \ref{regge:af:bd} with  $\mathfrak m \neq 0$ we may attach a (Hamiltonian) {center of mass} by 
	\begin{equation}\label{cm:bd}
		\mathcal C^+_\alpha=\lim_{r\to +\infty}\frac{1}{16\pi \mathfrak m}\left(\int_{S^{2}_{r,+}}\mathbb U(x_\alpha,e^+)\left(\frac{x}{r}\right)dS^{2}_{r,+}-\int_{S^1_{r}}x_\alpha e^+\left(\frac{x}{r},\vartheta\right)dS^1_{r}\right), \quad \alpha=1,2. 
	\end{equation}  
	This invariant has been introduced in  \cite{de2019mass}.
	
	As in the boundaryless case, we may also consider the {\em half-Schwarzschild metric}
	$g_m^+=g_m|_{\mathbb R^3_{+}\backslash \{\vec 0\}}$. 
	
	\begin{definition}\label{half:shc:def} 
		An asymptotically flat $3$-manifold with a non-compact boudnary $(M,g,\Sigma)$
		is {\em asymptotically half-Schwarzschild}
		(ahS) if a neighborhood of infinity is diffeomorphic to the complement of a hemisphere in $\mathbb R^3_+$ so that 
		\begin{equation}\label{asym:sch:+}
			g=\left(1+\frac{2m}{r}\right)\delta^++p^+,\quad p^+= O_3(r^{-2}). 
		\end{equation}
	\end{definition}
	
	As usual, if we assume further that each $x_\alpha R_g$ and $x_\alpha H_\Sigma$ are integrable then the ${\rm RT}^+$ conditions are satisfied and the center of mass $\mathcal C^+$ is well defined for ahS manifolds.

	Theorem \ref{remarkisolargebd} suggests that for an ahS manifold with $\mathfrak m>0$, large coordinate hemispheres may be perturbed to yield global solutions of the corresponding {\em relative} isoperimetric problem, where each competing surface $S$ satisfies $\partial S\subset\Sigma$ and ${\rm int}\,S\cap \Sigma=\emptyset$, with the relevant
	constrained  volume being the one  enclosed by $S$ and $\Sigma$. 
	As a first step towards this goal we establish here the following result:
	
	\begin{theorem}\label{free:af:bd}
		Assume that $(M,g,\Sigma)$ is an ahS   $3$-manifold with a non-compact boundary $\Sigma$ and satisfying the ${\rm RT}^+$ conditions. If ${\mathfrak m}=m/2 >0$ then there exists a neighborhood of infinity which is foliated by   {strictly stable} free boundary CMC hemispheres.
		Moreover, the geometric center  of this  foliation coincides with the Hamiltonian center of mass $\mathcal C^+$ of $(M,g,\Sigma)$.
	\end{theorem}
	
	The proof of this result is presented in Section \ref{stab:fol}. Again, the stability statement above should be interpreted in the sense of Appendix \ref{var:setup}.
	
	Our next result solves the relative isoperimetric problem referred to above by extending a celebrated result due to Eichmair-Metzger \cite{eichmairlarge} to our setting. 
	To state it, recall that the relative isoperimetric profile $I_g:[0,+\infty)\to [0,+\infty)$ of $(M,g,\Sigma)$ is given by 
	\[
	I_g(V)=\inf_{\Omega}\, \mathcal P(\partial^*\Omega,\interior M),
	\]
	where $\interior M=M\backslash \Sigma$ is the interior of $M$ and $\partial^*\Omega$ is the relative reduced boundary of a Borel set $\Omega\subset M$ satisfying: 
	
	i) ${\rm vol}(\Omega)=V$; \qquad
	ii) the relative perimeter $ P(\partial^*\Omega,\interior M)$ is finite. 
	
	We stress that in the computation of $\mathcal P(\partial^*\Omega,\interior M)$ only that part of the boundary area {\em inside} $\interior M$ counts, hence the qualification ``relative''. A solution to the relative isoperimetric problem is a Borel subset $\Omega$ such that ${\rm vol}(\Omega)=V$ and $\mathcal P(\partial^*\Omega,\interior M)=I_g(V)$ for some $V>0$. We then say that $\Omega$ is a (relative) isoperimetric region and its boundary $\partial \Omega\cap \interior M$ is a (relative) isoperimetric surface.

	\begin{theorem}\label{iso:ext:em1}
		Let $(M,g,\Sigma)$ be as in Theorem \ref{free:af:bd}. Then there exists $V_0>0$ such that for any $V\geq V_0$  a bounded isoperimetric region with volume $V$ exists. That region has a connected and smooth relative boundary which remains close to a centered coordinate hemisphere. Moreover, it sweeps out the whole manifold as  $V\to +\infty$. In particular, the corresponding  isoperimetric surfaces coincide with the leaves of the foliation in Theorem \ref{free:af:bd}, thus being unique for each value of the enclosed volume.  
	\end{theorem}
	
	This result, whose proof is presented in Section \ref{large:rel:iso}, provides a very precise description of the relative isoperimetric
	profile of $(M,g,\Sigma)$ as above for all sufficiently large volume values.

	%%%%%%%%%%%%%%%%%%%%%%%%%%%%%%%%%%%%%%%%%%%%%%%%%%%%%%%%%%%%%%%%%%%%%%%%%%%%%%%%%%%%%%%%%%%%%%%%%%%%%%%%%%%%%%%%%%%%%%%%%%%%%%%%%%%%%%%%%%%%%%%%%%%%%%%

	\section{The foliations in the boundaryless case: the proof of Theorem \ref{fol:curv:cond}}\label{const:fol}
	
	The purpose of this section is to explain how the well-known implicit function method presented in \cite{ye1997foliation,huang2008center} may be adapted to prove Theorem \ref{fol:curv:cond}.   
	We consider, in an aSRT manifold as in that theorem, the coordinate sphere $S^2_{\rho}(a)$ of radius $\rho>0$  centered at some $a\in\mathbb R^3$ to be chosen later. A justifiable assumption here is that $a$ varies in a bounded region and we will take it for granted. 
	In particular, we will sometimes ommit the dependence on $a$ in the estimates below.

	The following two results play a central role in our argument.
	
	\begin{proposition}\cite[Equation (5.1)]{huang2008center}\label{mc:huang}
		In an aS manifold, the mean curvature of  $S^2_{\rho}(a)$ at a point $x$ is 
		\begin{equation}\label{mc:huang:f}
			H_{\rho,a}(x)=\frac{2}{\rho}-\frac{4m}{\rho^2}+\frac{9m^2}{\rho^3}+\frac{6m(x-a)\cdot a}{\rho^4}+G_{\rho,a}(x)+O(\rho^{-4}),
		\end{equation}
		where 
		\begin{eqnarray}\label{gra:eq}
			G_{\rho,a}(x) & = & \frac{1}{2}p_{ij,k}(x)\mathfrak r_i \mathfrak r_j\mathfrak r_k+2\frac{p_{ij}(x)}{\rho}\mathfrak r_i\mathfrak r_j-p_{ij,i}(x)\mathfrak r_j\nonumber\\
			& & \quad -\frac{p_{ii}(x)}{\rho}+\frac{1}{2}p_{ii,j}(x)\mathfrak r_j,
		\end{eqnarray}
		with $p=g-\left(1+\frac{2m}{r}\right)\delta$ and $\mathfrak r=(x-a)/\rho$.
	\end{proposition}
	
	\begin{proposition}\cite[Lemma 5.1]{huang2008center}\label{mc:huang:mean}
		The center of mass $\mathcal C$ of an aSRT manifold satisfies  
		\begin{equation}\label{mc:huang:mean2}
			\int_{S^2_\rho(a)}(x_i-a_i)G_{\rho,a}(x)dS_\rho^{2,\delta}(a)=-8\pi m\mathcal C_i+O(\rho^{-1}), \quad i=1,2,3,
		\end{equation}
		where $dS^{2,\delta}_\rho(a)$ is the area element induced by the flat metric. 
	\end{proposition}
	
	Let $(M,g)$ be an aSRT $3$-manifold as in Theorem \ref{fol:curv:cond} and assume that it is also $\epsilon$-aS  with $\epsilon\geq 0$ and some fixed $c\in\mathbb R^3$ (recall that $c$ is irrelevant for the case $\epsilon=0$).
	The key observation now is that we can express the Gauss-Kronecker curvature $K_{\rho,a}$ of $S^2_\rho(a)$ in terms of the mean curvature $H_{\rho,a}$ up to terms decaying fast enough. This will allow us to make use of Propositions \ref{mc:huang} and \ref{mc:huang:mean}.
	
	\begin{proposition}\label{k:func:h}
		There holds
		\begin{equation}\label{K:func:h:f}
			2\rho K_{\rho,a}=H_{\rho,a}\left(1-\frac{2m}{\rho}
			+\frac{9m^2-3\gamma_2}{2\rho^2}+\frac{3m}{\rho^3}x\cdot a-\frac{3\gamma_1}{2\rho^3}x\cdot c+ O(\rho^{-2-\epsilon})\right)+O(\rho^{-5}). 
		\end{equation}
	\end{proposition} 
	The proof of Proposition \ref{k:func:h}, which relies on the fact that the Schwarzschild space carries a radial conformal vector field,  is deferred to Appendix \ref{conf:field}.
	\begin{corollary}\label{k:func:h:c}
		There holds
		\begin{equation}\label{gk:huang:f}
			2\rho \widetilde K_{\rho,a}
			=
			\frac{2}{\rho}-\frac{6m}{\rho^2}+\frac{{26}m^2-3c\gamma_2}{\rho^3}+\frac{12m}{\rho^4}x\cdot a-\frac{3\gamma_1}{\rho^4}x\cdot c+ G_{\rho,a}(x)+{O(\rho^{-3-\epsilon})}. 
		\end{equation}
	\end{corollary}
	\begin{proof}
		Combine (\ref{K:func:h:f}), (\ref{mc:huang:f}) and the fact that 
		\begin{equation}\label{ricc:norm}
			{\rm Ric}_g(\nu,\nu)=-\frac{2m}{\rho^3}+{O}(\rho^{-4-\epsilon}). 
		\end{equation}
	\end{proof}

	To proceed we consider for a function $\phi\in C^{2,\alpha}(S^2_\rho(a))$ the corresponding normal graphical surface over $S^2_\rho(a)$:
	\begin{equation}\label{graph:surf}
		S^2_\rho(a,\phi)=\{x+\rho^{-\theta}\phi(x)\nu(x);x\in S^2_\rho(a)\},\quad \theta\in (0,1).
	\end{equation}
	By the Taylor's formula, the modified Gauss-Kronecker curvature $\widetilde K_\rho(a,\rho^{-\theta}\phi)$ of $S^2_\rho(a,\phi)$ expands as 
	\begin{equation}\label{exp:K}
		\widetilde K_\rho(a,\rho^{-\theta}\phi)=\widetilde K_\rho(a,0)+d\widetilde K_\rho(a,0) (\rho^{-\theta}\phi)+\int_0^1(1-s)d^2\widetilde K_\rho(a,s\rho^{-\theta}\phi)(\rho^{-\theta}\phi,\rho^{-\theta}\phi)ds.
	\end{equation}
	We now observe that $\widetilde K_\rho(a,0)=\widetilde K_{\rho,a}$ and $d\widetilde K_\rho(a,0)=L_{S^2_\rho(a)}$ is the Jacobi operator appearing in (\ref{stab:op:gk}), whose asymptotic behavior we need to determine. With this goal in mind  
	we introduce local coordinates $\{y_1,y_2\}$ on $S^2_\rho(a)$ and let $\delta_\rho$ be the induced {\em Euclidean} metric, so that $h=g|_{S^2_\rho(a)}$ is given by
	\[
	h=\left(1+\frac{2m}{\rho}\right)\delta_\rho+O(\rho^{-2}). 
	\]

	\begin{proposition}\label{stabop:asymp}
		One has 
		\[
		L_{S^2_\rho(a)}=-\frac{1}{\rho}\left(\Delta_{\rho}+\frac{2}{\rho^2}\right)+O(\rho^{-4}), 
		\]
		where $\Delta_{\rho}$ is the Laplacian with respect to $\delta_\rho$. 
	\end{proposition}
	
	\begin{proof}
		With the notation of Appendix \ref{var:setup} we have
		\[
		\Lambda_{S^2_\rho(a)}=\frac{1}{\sqrt{\det h}}\partial_B\left(\sqrt{\det h}\,h^{AC}\Pi^B_A\partial_C\right), \quad \partial_A=\partial/\partial y_A.
		\]
		Since 
		\[
		\sqrt{\det h}=\left(1+\frac{2m}{\rho}\right)\sqrt{\det \delta_\rho}+O(\rho^{-2}), \quad 
		h^{AC}=\left(1+\frac{2m}{\rho}\right)^{-1}\delta_\rho^{AC}+O(\rho^{-2}),
		\] 
		and 
		\begin{equation}\label{pos:pi}
			\Pi_A^B=\rho^{-1}\delta^B_A+O(\rho^{-2}),
		\end{equation}
		we compute that
		\begin{eqnarray*}
			\Lambda_{S^2_\rho(a)} & = & \frac{1}{\sqrt{\det h}}\partial_B\left(\rho^{-1}\sqrt{\det \delta_\rho}\,\delta_\rho^{BC}\partial_C+O(\rho^{-1})\right)\\
			& = & \frac{1}{\rho}\Delta_{\rho}+O(\rho^{-4}).
		\end{eqnarray*}
		On the other hand, from (\ref{mc:huang:f}) and (\ref{gk:huang:f}),
		\[
		H_{\rho,a}K_{\rho,a}=\frac{2}{\rho^3}+O(\rho^{-4}). 
		\]
		Also, from (\ref{ricc:norm}), (\ref{pos:pi}) and the fact that 
		\begin{equation}\label{Riem:nor}
			{\rm Riem}_g^\nu=\frac{1}{2}{\rm Ric}_g(\nu,\nu)h+O(\rho^{-4}),
		\end{equation} 
		we see that 
		\begin{equation}\label{be:improved}
			\frac{1}{2}(\nabla_\nu{\rm Ric}_g)(\nu,\nu)-{\rm tr}_h(\Pi{\rm Riem}_g^\nu)=O(\rho^{-4}).
		\end{equation}
		The result follows. 
	\end{proof}
	
	We now have at our disposal the ingredients needed to prove the existence of a foliation satisfying (\ref{curcond}) in Theorem \ref{fol:curv:cond}. Indeed, {from (\ref{gk:huang:f}) and (\ref{exp:K})} we see that finding $\phi$ so that 
	\begin{equation}\label{const:gk}
		\widetilde K_\rho(a,\rho^{-\theta}\phi)=\frac{1}{\rho^2}-\frac{3m}{\rho^3}
	\end{equation}
	is equivalent to solving 
	\begin{equation}\label{eq:phi:R}
		2\rho^{-\theta}\Delta_{(\rho)}\phi(x)
		=\frac{26m^2-3\gamma_2}{\rho^3}+\frac{12m}{\rho^4}x\cdot a-\frac{3\gamma_1}{\rho^4}x\cdot c
		+G_{\rho, a}(x)+E_{\rho,\phi}(x),
	\end{equation}
	where $\Delta_{(\rho)}=\Delta_\rho+{2}{\rho^{-2}}$ and the remainder term $E_{\rho,\phi}$ is controlled as
	\begin{equation}\label{control:E}
		E_{\rho,\phi}(x)=O\left(\rho^{-3-\theta}|\phi|+\rho^{-3-2\theta}|\phi|^2
		+\rho^{-1-2\theta}|\phi||\partial^2\phi|\right)+O(\rho^{-3-\epsilon}). 
	\end{equation}
	
	We next pull back this equation under the map $F:S^2_1(\vec 0)\to S^2_\rho(a)$, $F(\mathfrak r)=a+\rho \mathfrak r$, so as to obtain an equation for $\psi=F^*\phi$ on $S^2_1(\vec 0)$:
	\begin{equation}\label{eq:phi:1:a}
		2\Delta_{(1)}\psi(\mathfrak r)
		=\mathscr F(\mathfrak r,a,\psi)
	\end{equation}
	where $\Delta_{(1)}=\Delta_1+{2}$ and
	\begin{align*}
		\mathscr F(\mathfrak r,a,\psi)=
		&\frac{26m^2-3\gamma_2}{\rho^{1-\theta}}
		+\frac{12m}{\rho^{1-\theta}}\mathfrak r\cdot a
		-\frac{3\gamma_1}{\rho^{1-\theta}}\mathfrak r\cdot c
		+\rho^{2+\theta}F^*G_{\rho,a}(\mathfrak r)
		\\
		&+\rho^{2+\theta}F^*E_{\rho,\phi}(\mathfrak r)
		+12m \rho^{-2+\theta}|a|^2
		-3\gamma_1\rho^{-2+\theta}a\cdot c.
	\end{align*}
	
	We note that the primary obstruction to solving (\ref{eq:phi:1:a}) is the fact that the operator 
	$$\Delta_{(1)}:C^{2,\alpha}(S^2_1(\vec 0))\to C^\alpha(S^2_1(\vec 0))$$ 
	has a nontrivial cokernel generated by the functions $\mathfrak{r}_i$, $i=1,2,3$. Thus, in order to remove this obstruction, we shall calculate
	$$
	\int_{S^2_1(\vec 0)}\mathfrak r_i\mathscr F(\mathfrak r, a,\psi) \,dS^{2,\delta}_1(\vec 0).
	$$
	Note that by symmetry
	$$
	\int_{S_1(\vec 0)}\mathfrak r_i
	\left(
	\frac{26m^2-3\gamma_2}{\rho^{1-\theta}}
	+12m \rho^{-2+\theta}|a|^2
	-3\gamma_1\rho^{-2+\theta}a\cdot c\right)
	dS^{2,\delta}_1(\vec 0)=0.
	$$
	Using (\ref{mc:huang:mean2}) we easily see that 
	\begin{align*}
		\int_{S^2_1(\vec 0)}&\mathfrak{r}_i\left(\frac{12m}{\rho^{1-\theta}} \mathfrak{r}\cdot a
		-\frac{3\gamma_1}{\rho^{1-\theta}} \mathfrak{r}\cdot c+\rho^{2+\theta}F^*G_{\rho,a}(\mathfrak r)\right)dS^{2,\delta}_1(\vec 0)
		\\
		& =  \frac{16\pi m}{\rho^{1-\theta}}\left(a_i-\frac{\gamma_1}{4m}c_i-\frac{\mathcal C_i}{2}\right)+O(\rho^{\theta-2})
		\\
		& = \frac{16\pi m}{\rho^{1-\theta}}\left(a_i-\frac{\gamma_1}{2m}c_i\right)+O(\rho^{\theta-2}),
	\end{align*}
	where in the last step we used that $\mathcal C=\frac{\gamma_1}{2m}c$, according to Remark \ref{asym:req}.
	The remaining integral 
	$$
	\int_{S_1(\vec 0)}\mathfrak r_i
	\rho^{2+\theta}F^*E_{\rho,\phi}(\mathfrak r)
	dS^{2,\delta}_1(\vec 0)
	$$
	may be estimated so as to yield
	\begin{equation}\label{remain:int}
		\int_{S^2_1(\vec 0)}\mathfrak r_i\mathscr F(\mathfrak r,a,\psi) dS^{2,\delta}_1(\vec 0)=
		\frac{16\pi m}{\rho^{1-\theta}}\left(a_i-\frac{\gamma_1}{2m}c_i\right)+\frac{1}{\rho^{1-\theta}}\widehat E_i,
	\end{equation}
	where
	\begin{equation}\label{remainder}
		\widehat E_i=O((\rho^{-1}+\rho^{-\theta})\|\psi\|_{C^2}){+O(\rho^{-\epsilon})}.
	\end{equation}
	Thus, for each $\rho$ large enough we may choose $a_\rho\in\mathbb R^3$ such that 
	\begin{equation}\label{choose:a}
		(a_\rho)_i=\frac{\gamma_1}{2m}c_i-\frac{1}{16\pi m}\widehat E_i,
	\end{equation}
	so as to have
	\begin{equation}\label{elim:obs}
		\int_{S^2_1(\vec 0)}\mathfrak r_i\mathscr F(\mathfrak r, a_\rho,\psi) \,dS^{2,\delta}_1(\vec 0)=0,\qquad i=1,2,3,
	\end{equation}
	for any $\psi$ with $\|\psi\|_{C^2}$ bounded.
	
	With the obstruction so removed we may now use a standard fixed point argument to check that (\ref{eq:phi:R}) has a {\em unique} solution $\phi_\rho$ for all such $\rho$. 
	More precisely,  from (\ref{elim:obs}) we see that $\mathscr F(\mathfrak r,a_\rho,\psi)$ lies in ${\rm Ran}\,\Delta_{(1)}$ if $\|\psi\|_{C^{2,\alpha}}\leq 1$. Therefore, we may uniquely solve 
	\begin{equation}\label{prop:fix:pt}
		2\Delta_{(1)}\widetilde \psi=\mathscr F(\mathfrak r,a_\rho,\psi),
	\end{equation}
	for $\widetilde \psi\in C^{2,\alpha}(S^2_1(\vec 0))\cap (\ker \Delta_{(1)})^\perp$ satisfying
	\begin{equation}\label{estim:theta}
		\|\widetilde \psi\|_{C^{2,\alpha}}\leq C\|\mathscr F(\mathfrak r,a_\rho,\psi)\|_{C^{0,\alpha}}\leq C'\rho^{\max\{-\theta,\theta-1,\theta-1-\epsilon\}}=C'\rho^{\max\{-\theta,\theta-1\}},
	\end{equation}
	and this is $\leq 1$ if $\rho$ is large enough. Thus, the map $\psi\mapsto \widetilde\psi$ has a fixed point which yields a solution $\psi_\rho$ of (\ref{eq:phi:1:a}) and hence a solution $\phi_\rho$ of (\ref{eq:phi:R}). 
	In particular, the graphical surface  associated to $\phi_\rho$, denoted simply 
	$S^2_\rho(\phi_\rho)$, has constant modified Gauss-Kroneker curvature given by the right-hand side of (\ref{const:gk}).
	Moreover, if we choose $\theta=1/2$ in (\ref{estim:theta}) then this analysis guarantees that $\phi_\rho\in C^{2,\alpha}(S^2_\rho(a_\rho))$ satisfies 
	\begin{equation}\label{schauder}
		\sum_{|I|\leq 2}\rho^{|I|}|\partial_I\phi_\rho|+\sum_{|I|= 2}\rho^{2+\alpha}[\partial_I\phi_\rho]_\alpha\leq C''\rho^{1/2}, 
	\end{equation}
	so the corresponding graphical surface $S^2_\rho(\phi_\rho)$, which actually involves the function $\rho^{-1/2}\phi_\rho$,  remains at a fixed distance of $S^2_\rho(a_\rho)$ while becoming rounder as $\rho\to+\infty$.  
	The geometric center of mass of the foliation is given by 
	\[ 
	\mathcal C_{\mathcal F}=\lim_{\rho\to+\infty}\frac{\int_{S^2_\rho(\phi_\rho)}z\,dS^{2,\delta}_\rho(\phi_\rho)}{\int_{S^{2}_\rho(\phi_\rho)}dS^{2,\delta}_\rho(\phi_\rho)}, 
	\]
	where $z$ is the position vector. We easily see from (\ref{remainder}), (\ref{choose:a}) and (\ref{schauder}) that either 
	$\mathcal C_{\mathcal F}$  remains at a finite distance from $\mathcal C$ if $\epsilon=0$ or there holds $\mathcal C_{\mathcal F}=\mathcal C$ if $\epsilon>0$ (recall that $\mathcal C=\frac{\gamma_1}{2m}c$). 
	Also, each such surface  may be viewed as a graph over $S^2_{\rho}(c)$. For simplicity of notation, we still denote such a surface by $S^{2}_\rho(\phi_\rho)$.
	Finally, for further use we note that by (\ref{schauder}) we may determine the asymptotic expansions of the geometric invariants of $S^2_\rho(\phi_\rho)$. The result is 
	\begin{equation}\label{est:asym:gr}
		\left\{
		\begin{array}{rcl}
			K & = & \rho^{-2} -4m\rho^{-3}+O(\rho^{-4}),\\
			H & = & {2}{\rho^{-1}}-{4m}{\rho^{-2}}+O(\rho^{-3}),\\
			|W|^2 & = & {2}{\rho^{-2}}-{8m}{\rho^{-3}}+O(\rho^{-4}),\\
			{\rm Ric}(\nu,\nu) & = & -{2m}{\rho^{-3}}+O(\rho^{-4}),\\
			K_G & = & {\rho^{-2}}- {2m}{\rho^{-3}}+O(\rho^{-4}).
		\end{array}
		\right.
	\end{equation}
	Here, $K_G$ is the Gaussian curvature.

	It remains to check that for $\rho_0$ large enough, the family of surfaces $S^2_\rho(\phi_\rho)$, $\rho\geq \rho_0$, defines a foliation whose leaves are strictly stable in the appropriate sense. We first tackle the stability issue.  
	
	\begin{theorem}\label{stab:fol:th}
		If $m>0$ then $S^2_\rho(\phi_\rho)$ is strictly stable for all $\rho$ large enough.
	\end{theorem} 
	
	\begin{proof}
		According to Proposition \ref{stable:cond}, we must estimate from below the quadratic form 
		\[
		V(f)=
		\int_S\left( \langle\Pi\nabla_S f,\nabla_S f\rangle-f^2\left( HK+\frac{1}{2}(\nabla_\nu{\rm Ric}_g)(\nu,\nu)+{\rm tr}_S(\Pi{\rm Riem}_g^\nu)\right) \right)dS, 
		\]
		where $f\in \mathcal G(S)$ and here we set $S=S^2_\rho(\phi_\rho)$ for simplicity.
		From (\ref{est:asym:gr}) we have 
		\[
		-HK=-\frac{2}{\rho^3}+\frac{12m}{\rho^4}+O(\rho^{-5}).
		\]
		Also, from   (\ref{ricc:norm}), (\ref{pos:pi}) and (\ref{Riem:nor}), we may improve (\ref{be:improved}) to
		\[
		-\frac{1}{2}(\nabla_\nu{\rm Ric}_g)(\nu,\nu)-{\rm tr}_S(\Pi{\rm Riem}_g^\nu)={\color{blue}{-}}\frac{m}{\rho^4}+O(\rho^{-5}).
		\]
		Thus, 
		\begin{equation}\label{est:quad:V}
			V(f)\geq 
			\int_S \langle\Pi\nabla_S f,\nabla_S f\rangle dS+
			\left(-\frac{2}{\rho^3}+\frac{11m}{\rho^4}+O(\rho^{-5})\right)\int_S f^2 dS. 
		\end{equation}
		We now observe that  the Newton tensor of $S=S^2_\rho(\phi_\rho)$ satisfies
		\begin{equation}\label{exp:Pi}
			\Pi=\left(\frac{1}{\rho}-\frac{2m}{\rho^2}\right)I+O(\rho^{-3}). 
		\end{equation}
		On the other hand, by the well-known Lichnerowicz eigenvalue bound, 
		\[
		\int_S|\nabla_Sf|^2dS\geq 2\inf K_G\int_Sf^2dS, \quad f\in\mathcal G(S).
		\]
		where by (\ref{est:asym:gr}), 
		\[
		\inf K_G\geq \frac{1}{\rho^2}-\frac{2m}{\rho^3}-C\rho^{-4}, \quad C>0.
		\]
		We thus conclude that  
		\[
		V(f)\geq \left(\frac{3m}{\rho^4}-C\rho^{-5}\right)\int_Sf^2dS,
		\]
		and the result follows.
	\end{proof}
	
	We now check that the surfaces define a foliation. 
	Since the argument, as explained for instance in \cite{huang2009center}, is well-known by now and it  may be easily adapted to our setting, here we merely sketch the proof. 
	
	\begin{proposition}\label{unconst:eig}
		Let $\zeta_0<\zeta_1$ be the first two (unconstrained) eigenvalues of $L_{S_\rho^2(\phi_\rho)}$. Then 
		\[
		\zeta_0=-\frac{2}{\rho^3}+\frac{9m}{\rho^4}+O(\rho^{-5}), \quad \zeta_1\geq \frac{m}{\rho^4}-C\rho^{-5}. 
		\]
		In particular, $L_{S_\rho^2(\phi_\rho)}:C^{2,\alpha}(S_\rho^2(\phi_\rho))\to C^{\alpha}(S^2_\rho(\phi_\rho))$ is invertible for all $\rho$ large enough. 
	\end{proposition}

	The proof of this statement is basically a refinement of the stability analysis above. 
	In any case, for any such fixed $\rho_0$  large enough, it implies the existence of a {\em unique} $f_{\rho_0}\in C^{2,\alpha}(S_{\rho_0}^2(\phi_{\rho_0}))$ such that $L_{S_{\rho_0}^2(\phi_{\rho_0})}f_{\rho_0}=1$. 
	
	\begin{proposition}\label{nonvan}
		The function $f_{\rho_0}$ vanishes nowhere on $S^2_{\rho_0}(\phi_{\rho_0})$. 
	\end{proposition}

	\begin{proof}
		This  follows immediately from the estimate
		\[
		\sup_{S_{\rho_0}^2(\phi_{\rho_0})} |f_{\rho_0}-\overline{f_{\rho_0}}|\leq C\rho_{0}^{-1}|\overline{f_{\rho_0}}|,
		\]
		where $C>0$ is a constant depending only $g$ and the overline stands for the average over $S_{\rho_0}^2(\phi_{\rho_0})$. The method of proof, which is explained in \cite[Section 5.2]{huang2009center}, makes use of Nash-Moser iteration and equally applies here due to the fact that $L_{S_{\rho_0}^2(\phi_{\rho_0})}$ is elliptic of divergence type; see Appendix \ref{var:setup}.
	\end{proof}
	
	Now, we may organize the graphical surfaces $S_{\rho_0}^2(\phi_\rho)$ with $\rho$ close to $\rho_0$ in a smooth deformation 
	\[
	F:(\widetilde K_{\rho_0}-\epsilon,\widetilde K_{\rho_0}+\epsilon)\times S_{\rho_0}^2(\phi_{\rho_0})\to M,\quad \epsilon>0,
	\] 
	where $\widetilde K_{\rho_0}$ is the modified Gauss-Kronecker curvature  of $S_{\rho_0}^2(\phi_{\rho_0})$  and $F(\widetilde K,\cdot)$ has constant modified Gauss-Kronecker curvature equal to $\widetilde K$. Clearly,
	\[
	F(\widetilde K, x)=\exp_{S_{\rho_0}^2(\phi_{\rho_0})}(\widetilde f_{\widetilde K}\nu),
	\]
	for some function $\widetilde f_{\widetilde K}$ on $S_{\rho_0}^2(\phi_{\rho_0})$ with $\widetilde f_{\widetilde K_{\rho_0}}\equiv 0$. 
	Let 
	\[
	\widetilde f_{0}:=\frac{\partial \widetilde f_{\widetilde K}}{\partial \widetilde K}|_{\widetilde K=\widetilde K_{\rho_0}}=\left\langle \frac{\partial F}{\partial \widetilde K}|_{\widetilde K=\widetilde K_{\rho_0}},\nu\right\rangle.
	\]
	By (\ref{varsig2tilde}), 
	\[
	L_{S_{\rho_0}^2(\phi_{\rho_0})}\widetilde f_{0}=\frac{d}{d k}\left(\widetilde K_{\rho_0}+k\right)|_{k=0}=1,
	\]
	so that $\widetilde f_{0}=f_{\rho_0}$ by uniqueness. In particular, $\widetilde f_{0}$ never vanishes and we may use the inverse function theorem to conclude that $F$ is a diffeomorphism onto a small neighborhood of $S_{\rho_0}^2(\phi_{\rho_0})$ in $M$.  
	This shows that the surfaces define a foliation and completes the proof of the first part of Theorem \ref{fol:curv:cond}.  
	
	We now sketch the proof of the existence of a stable foliation satisfying the curvature condition in (\ref{curcond2}), whose leaves 
	correspond to surfaces extremizing the total mean curvature under an area constraint by Appendix \ref{var:setup}.  
	The  Lagrange multiplier $\gamma$ is determined by observing that  (\ref{K:func:h:f}) leads to 
	\[
	\frac{\widetilde K_{\rho,a}}{H_{\rho,a}}=\frac{1}{2\rho}-\frac{m}{2\rho^2}+O(\rho^{-3}).
	\]
	With this choice of $\gamma=O(\rho^{-1})>0$, it then follows from (\ref{exp:Pi}), (\ref{varsig2tilde}) and (\ref{evolmean}) that the quadratic form associated to the linearization of (\ref{curcond2}) at $S=S_\rho(a)$ has 
	\[
	\int_S\langle\left(\Pi-\gamma I\right)\nabla_Sf,\nabla_S f \rangle dS=\left(\frac{1}{2\rho}-\frac{3m}{2\rho^2}+O(\rho^{-3})\right)\int_S|\nabla_Sf|^2dS
	\]
	as its principal part. Hence, the linearization is selfadjoint and elliptic. In fact, a computation shows that the rescaled linearization on $S^2_1(\vec 0)$ is $2(1-\gamma)\Delta_{(1)}$; compare with the left-hand side of (\ref{eq:phi:1:a}). Thus, we are in a position to run the implicit function method above in order to  construct a graphical surface over $S_\rho(a)$ satisfying (\ref{curcond2}). As in the preceding case, this step only uses that $m\neq 0$. If $m>0$ then a further analysis shows that these graphical surfaces comprise a foliation of a neighborhood of infinity whose leaves are strictly stable in the appropriate sense. Moreover, the geometric center of this foliation {remains at a finite distance} from the Hamiltonian center of mass. In this way,
	the proof of  Theorem \ref{fol:curv:cond} is completed.

	%%%%%%%%%%%%%%%%%%%%%%%%%%%%%%%%%%%%%%%%%%%%%%%%%%%%%%%%%

	\section{The foliation by free boundary CMC hemispheres: the proof of Theorem \ref{free:af:bd}}\label{stab:fol}

	In this section we present the proof of  Theorem \ref{free:af:bd}. 
	{In any ahS manifold as in that theorem we consider the coordinate hemisphere $S^2_{\rho,+}(b)$ centered at some $b\in\mathbb R^2$ and use the notation of Appendix \ref{var:setup}}.  As in Proposition \ref{mc:huang},
	we compute the mean curvature $H_{\rho,+,b}$ of this hemisphere to obtain
	\begin{equation}\label{mc:huang:f:b}
		H_{\rho,+,b}(x)=\frac{2}{\rho}-\frac{4m}{\rho^2}+\frac{9m^2}{\rho^3}+\frac{6m(x-b)\cdot b}{\rho^4}
		+G_{\rho. b}(x)+O(\rho^{-4}),
	\end{equation}
	where $G_{\rho. b}(x)$ is as in (\ref{gra:eq}) with $\mathfrak r=(x-b)/\rho$. We also need the analogue of Proposition \ref{mc:huang:mean}, which goes as follows.
	
	\begin{proposition}\label{mc:huang:mean:b}
		If $(M,g,\Sigma)$ is an ahS manifold meeting the ${\rm RT}^+$ condition then its center of mass $\mathcal C^+$ satisfies 
		\begin{equation}\label{mc:huang:mean:b2}
			\int_{S^2_{\rho,+}(b)}(x_\alpha-b_\alpha)G_{\rho. b}(x) dS^{2,\delta_+}_{\rho,+}(b)=-8\pi \mathfrak m\mathcal C^+_\alpha+O(\rho^{-1}), \quad \alpha=1,2, 
		\end{equation}
		where $dS^{2,\delta_+}_{\rho,+}(b)$ is the area element induced by the flat metric. 
	\end{proposition}
	
	The proof of this proposition, which adapts an argument first appearing in \cite[Appendix F]{eichmair2013unique}, is presented in Appendix \ref{dens:res}.
	
	We now proceed to the proof of Theorem \ref{free:af:bd} via the standard implicit function method \cite{ye1997foliation, huang2008center}. We consider, for a function $\phi\in C^{2,\alpha}(S^2_{\rho,+}(b))$ satisfying {\em  Neumann boundary condition} along $S^1_\rho=\partial S^2_{\rho,+}$, the corresponding normal graphical surface over $S^2_{\rho,+}(b)$:
	\begin{equation}\label{normal:free}
		S^2_{\rho,+}(b,\phi)=\{x+\rho^{-\theta}\phi(x)\nu(x);\,x\in S^2_{\rho,+}(b)\},\quad \theta\in (0,1).
	\end{equation}
	By the Taylor's formula, the mean curvature $H_{\rho,+}(b,\rho^{-\theta}\phi)$ of $S^2_{\rho,+}(b,\phi)$ expands as 
	\[
	{{H_{\rho,+}(b,\rho^{-\theta}\phi)= H_{\rho,+}(b,0)+d H_{\rho,+}(b,0)(\rho^{-\theta}\phi)+\int_0^1(1-s)d^2 H_{\rho,+}(b,s\rho^{-\theta}\phi)(\rho^{-\theta}\phi, \rho^{-\theta}\phi) ds.}}\]
	We now observe that $H_{\rho,+}(b,0)= H_{\rho,+,b}$ and $dH_{\rho,+}(a,0)= \mathscr L_{S^2_{\rho,+}(b)}$ is the Jacobi operator appearing in (\ref{stab:op:0}).
	Thus, finding $\phi$ so that 
	\begin{equation}\label{const:h:const}
		H_{\rho,+}(b, \rho^{-\theta}\phi)=\frac{2}{\rho}-\frac{4m}{\rho^2}
	\end{equation}
	is equivalent to solving 
	\begin{equation}\label{eq:phi:R:2}
		\rho^{-\theta}\Delta_{(\rho)}\phi=\frac{9m^2}{\rho^3}+\frac{6m(x-b)\cdot b}{\rho^4}+{{G_{\rho,b}(x)}}+E_{\rho,\phi+}(x),
	\end{equation}
	where as usual $\Delta_{(\rho)}=\Delta_\rho+{2}{\rho^{-2}}$ (recall that  $\Delta_\rho$ is the Laplacian with respect to $\delta_\rho^+$, the induced round metric on $S^2_{\rho,+}$)
	and the remainder $E_{\rho,\phi+}$ has the {the same bound as in (\ref{control:E}).
		We next pull back this equation under the map $F:S^2_{1,+}(0)\to S^2_{\rho,+}(b)$, $F(\mathfrak r)=b+\rho \mathfrak r$, so as to obtain an equation for $\psi=F^*\phi$ on $S^2_{1,+}(0)$:
		\begin{equation}\label{eq:phi:1:a:2}
			\Delta_{(1)}\psi=\frac{9m^2}{\rho^{1-\theta}}+\frac{6m \mathfrak r\cdot b}{\rho^{1-\theta}}+\rho^{2+\theta} F^*G_{\rho,b}(\mathfrak r)+\rho^{2+\theta}F^*E_{\rho,\phi+},
		\end{equation}
		where $\Delta_{(1)}=\Delta_1+{2}$.
		
		We now recall  that the operator  $\Delta_{(1)}:C_\star^{2,\alpha}(S^2_{1,+}(\vec 0))\to C^\alpha(S^2_{1,+}(\vec 0))$, where the star means that we impose the {Neumann boundary condition}, has a nontrivial cokernel generated by the functions $\mathfrak r_\alpha$, $\alpha=1,2$. Clearly, this poses an obstruction to solving (\ref{eq:phi:1:a:2}). However, using (\ref{mc:huang:mean:b2}) we easily calculate that 
		\begin{eqnarray*}
			\int_{S^2_{1,+}(\vec 0)}\mathfrak r_\alpha\left(\frac{9m^2}{\rho^{1-\theta}}+\frac{6m \mathfrak r\cdot b}{\rho^{1-\theta}}+\rho^{2+\theta}F^*G_{\rho,b}(\mathfrak r)\right)dS_{1,+}^{2,\delta^+}(\vec 0)
			& = & \frac{8\pi \mathfrak m }{\rho^{1-\theta}}\left(b_\alpha-\mathcal C_\alpha^+\right)\\
			& & \quad +O(\rho^{-1}\|\psi\|_{C^2}),
		\end{eqnarray*}
		so we end up with 
		\[
		\int_{S^2_{1,+}(\vec 0)}\mathfrak r_\alpha \mathscr G(\mathfrak r,b,\psi) dS_{1,+}^{2,\delta^+}(\vec 0)=\frac{8\pi\mathfrak m}{\rho^{1-\theta}}(b_\alpha-\mathcal C^+_\alpha)+\frac{1}{\rho^{1-\theta}}\widehat E_\alpha,
		\]
		where $\mathscr G(\mathfrak r,b,\psi)$ is 
		the right-hand side of (\ref{eq:phi:1:a:2}) and 
		\begin{equation}\label{remainder:2}
			{{\widehat E_\alpha=O((\rho^{-1}+\rho^{-\theta})\|\psi\|_{C^2})}}
		\end{equation}
		Thus, for each $\rho$ large enough we may choose $b_\rho$ such that 
		\begin{equation}\label{choose:a:2}
			(b_\rho)_\alpha=\mathcal C^+_\alpha-\frac{1}{8\pi\mathfrak m}\widehat E_\alpha,
		\end{equation}
		so as to have
		\begin{equation}\label{elim:obs:2}
			\int_{S^2_{1,+}(\vec 0)}\mathfrak r_\alpha \mathscr G(\mathfrak r,b_\rho,\psi) \,dS_{1,+}^{2,\delta^+}(\vec 0)=0, 
			\qquad \alpha=1,2,
		\end{equation}
		for any $\psi$ with $\|\psi\|_{C^2}$ bounded. This eliminates the obstruction mentioned earlier.
		
		As in the proof of Theorem \ref{fol:curv:cond} above, we may now use the standard fixed point argument to check that (\ref{eq:phi:R:2}) has a {\em unique} solution $\phi_\rho$ for all such $\rho$. In particular, the  graphical surface corresponding to $\phi_\rho$ as in (\ref{normal:free}), denoted $S_{\rho,+}(\phi_\rho)$,  has constant mean curvature given by the right-hand side of (\ref{const:h:const}). Also, the Neumann condition imposed on $\phi_\rho$ implies that this graphical  surface is free boundary. 
		Moreover, an estimate similar to (\ref{schauder}) also holds true here, so $S^2_{\rho,+}(\phi_\rho)$ remains at a fixed distance
		of $S^2_{\rho,+}(b_\rho)$  
		while becoming rounder as $\rho\to+\infty$.  
		
		The geometric center of mass of this family of surfaces, 
		is given by 
		\[
		\mathcal C_{H}^+=\lim_{\rho\to+\infty}\frac{\int_{S^2_{\rho,+}(\phi_\rho)}z\,dS^{2,\delta^+}_{\rho,+}(\phi_\rho)}{\int_{S^{2}_{\rho,+}(\phi_\rho)}dS^{2,\delta^+}_{\rho,+}(\phi_\rho)}, 
		\]
		where $z$ is the position vector. Clearly, 
		$\mathcal C_{H}^+=\mathcal C^+$. 
		Also, each such surface  may be viewed as a graph over $S^2_{\rho}(\mathcal C^+)$. For simplicity  we retain the notation and still denote such a surface by $S^{2}_{\rho,+}(\phi_\rho)$.
		It remains to check that this family of free boundary CMC hemispheres defines a foliation of a neighborhood of infinity with stable leaves. As usual we first consider the stability issue.
		
		\begin{theorem}\label{stab:fb}
			If $\mathfrak m=m/2>0$ then $S_{\rho,+}(\phi_\rho)$ is strictly stable for all $\rho$ large enough.
		\end{theorem} 
		
		For the proof we first note that the geometric invariants of $S_{\rho,+}(\phi_\rho)$ expand as  
		\begin{equation}\label{est:asym:gr:fb}
			\left\{
			\begin{array}{rcl}
				|W|^2 & = & {2}{\rho^{-2}}-{8m}{\rho^{-3}}+O(\rho^{-4}),\\
				{\rm Ric}(\nu,\nu) & = & -{2m}{\rho^{-3}}+O(\rho^{-4}),\\
				K_G & = & {\rho^{-2}}- {2m}{\rho^{-3}}+O(\rho^{-4}).
			\end{array}
			\right.
		\end{equation}
		We also need the asymptotic expansion of the second fundamental of $\Sigma$, the non-compact boundary of $M$.

		\begin{lemma}\label{sff:exp}
			The second fundamental form $\mathcal B$ of $\Sigma$ satisfies 
			\begin{equation}\label{sff:exp:2}
				\mathcal B_{\alpha\beta}=O(\rho^{-3}), \quad \alpha,\beta=1,2.
			\end{equation}
		\end{lemma}
		
		\begin{proof}
			Recall that 
			\[
			g=\left(1+\frac{2m}{r}\right)\delta^++p^+, \quad p^+=O(\rho^{-2}).  
			\]
			Outside a compact subset of $M$, $\Sigma$ is defined by $x_3=0$  and its tangent space is generated by  $\{\partial_1,\partial_2\}$. If  $\eta=\eta^i\partial_i$ is the {{inward}} unit normal along $\Sigma$ then 
			\[
			\mathcal B_{\alpha\beta}=\langle\eta,\nabla_{\partial_\alpha}\partial_\beta\rangle=\Gamma_{\alpha\beta}^i\langle \eta,\partial_i
			\rangle=\Gamma_{\alpha\beta}^3g_{3i}\eta^i.
			\]
			Since 
			\begin{eqnarray*}
				\Gamma^3_{\alpha\beta} & = & {{\frac{1}{2}g^{33}\left(g_{\alpha 3,\beta}+g_{\beta 3,\alpha}-g_{\alpha\beta,3}\right)+O(\rho^{-4})}}\\
				& = & -\frac{1}{2}\left(1+\frac{2m}{r}\right)^{-1}g_{\alpha\beta,3}+O(\rho^{-3}),
			\end{eqnarray*}
			and 
			\begin{eqnarray*}
				g_{\alpha\beta,3} &  = & -2mr^{-2}\frac{\partial r}{\partial x_3}\delta_{\alpha\beta}+O(\rho^{-3})\\
				& = & -2mr^{-3}x_3\delta_{\alpha\beta}+O(\rho^{-3})\\
				& = & O(\rho^{-3}),
			\end{eqnarray*}
			the result follows. 
		\end{proof}
		
		By Proposition \ref{stab:cmcfree}, the proof of Theorem \ref{stab:fb} involves  
		estimating from below the quadratic form 
		\[
		Q(f)= \int_S\left(|\nabla_Sf|^2-\left(|W|^2+{\rm Ric}(\nu,\nu)\right)f^2\right)dS-\int_{\partial S}\kappa f^2d\partial S, \quad f\in \mathcal F(S), 
		\]
		where here we set $S=S_{\rho,+}(\phi_\rho)$ for simplicity. 
		We may assume that $\int_S f^2dS=1$, which implies {$f=O(\rho^{-1})$}. Hence, using (\ref{est:asym:gr:fb}), 
		\begin{equation}\label{quad:form}
			Q(f)=\int_S(|\nabla_Sf|^2dS-\int_{\partial S}\kappa f^2d\partial S-\frac{2}{\rho^2}+\frac{10m}{\rho^3}+O(\rho^{-4}).
		\end{equation}
		Thus, we are left with the task of estimating from below the quadratic form
		\[
		\widehat Q(f)=\int_S(|\nabla_Sf|^2dS-\int_{\partial S}\kappa f^2d\partial S, \quad f\in \mathcal F(S),
		\]
		which is equivalent to estimating from below the first eigenvalue  $\widehat \lambda$ of the eigenvalue problem
		\begin{equation}\label{eingen:prob}
			\left\{
			\begin{array}{ll} 
				-\Delta_Sf=\lambda f & {\rm in}\,\,\, S \\
				\frac{\partial f}{\partial \mu}=\kappa f & {\rm on}\,\,\, \partial S
			\end{array}
			\right.
		\end{equation}
		where $f\in \mathcal F(S)$.
		Notice that a comparison with the first eigenvalue  $2/\rho^2$ of the Neumann ($\kappa=0$) eigenvalue problem on $(S_{\rho,+}(a),\delta_\rho^+)\hookrightarrow\mathbb (R^3_+,\delta^+)$  already shows that $\widehat\lambda>0$ and provides the preliminary but useful estimate $\widehat\lambda=O(\rho^{-2})$.

		{If $f$ is an eigenfunction of (\ref{eingen:prob}) with eigenvalue $\lambda=\widehat\lambda$,}} we see that
	\begin{equation}\label{mult:eq}
		\int_S|\nabla_Sf|^2dS=\int_{\partial S}\kappa f^2d\partial S+\widehat\lambda. 
	\end{equation}
	From (\ref{sff:exp:2}), $\kappa=\mathcal B_{\alpha\beta}\nu^\alpha\nu^\beta=O(\rho^{-3})$, so that 
	$
	\int_{\partial S}\kappa f^2d\partial S=O(\rho^{-4})
	$
	and $\partial f/\partial\mu=O(\rho ^{-4})$, {{where \eqref{eingen:prob} was used in the latter step.}}
	Thus, 
	$
	\int_S|\nabla_Sf|^2dS=O(\rho^{-2}),
	$
	so that $\nabla_Sf=O(\rho^{-2})$ and, moreover, from (\ref{eingen:prob}) we get $\Delta_Sf=O(\rho^{-3})$. 
	
	We now apply a well-known integral identity due to Reilly \cite{reilly1977applications}. In our setting ($\dim S=2$) it simplifies to 
	\begin{eqnarray}\label{eq:reilly}
		\int_S\left((\Delta_Sf)^2-|\nabla_S^2f|^2\right)dS & = & 2\int_{\partial S}\frac{\partial f}{\partial \mu}\Delta_{\partial S }fd\partial S\\
		& & \quad +\int_{\partial S}H_{\partial S}\left(\left(\frac{\partial f}{\partial \mu}\right)^2+|\nabla_{\partial S}f|^2\right)d\partial S\notag\\
		& & \qquad +\int_SK_G|\nabla_Sf|^2dS,\notag
	\end{eqnarray}
	where $H_{\partial S}$ is the mean (in fact, geodesic) curvature  of $\partial S$ in $S$. Since $S$ is free boundary, $H_{\partial S}=\mathcal B(T,T)=O(\rho^{-3})$, where $T$ is a unit tangent vector along $\partial S$ and we used (\ref{sff:exp:2}). Thus, the second integral in the right-hand side equals
	\[
	\int_{\partial S}H_{\partial S}|\nabla_Sf|^2d\partial S=O(\rho^{-6}).
	\] 
	On the other hand,
	\[
	\Delta_{\partial S}f=\Delta_Sf-\frac{\partial^2f}{\partial\mu^2}=O(\rho^{-3}),
	\]
	so that the first integral in the right-hand side is also $O(\rho^{-6})$. Finally, using (\ref{est:asym:gr:fb}) and (\ref{mult:eq}), 
	\[
	\int_SK_G|\nabla_Sf|^2dS\geq \left(\frac{1}{\rho^2}-\frac{2m}{\rho^3}-C\rho^{-4}\right)\left(\widehat\lambda+O(\rho^{-4})\right), \quad C>0, 
	\]
	and combining this with the fact that 
	$|\nabla_S^2f|^2\geq (\Delta_Sf)^2/2$ and $\widehat\lambda=O(\rho^{-2})$  we get from \eqref{eq:reilly} that
	\[
	\widehat\lambda\left({\widehat\lambda}-\frac{2}{\rho^2}+\frac{4m}{\rho^3}\right)\geq -C'\rho^{-6}, \quad C'>0.
	\]
	Since we already know that $\widehat \lambda^{-1}=O(\rho^2)$ is positive, this gives
	\[
	\widehat\lambda\geq \frac{2}{\rho^2}-\frac{4m}{\rho^3}-C''\rho^{-4}, \quad C''>0.
	\]
	Combining this with (\ref{quad:form}) we finally have
	\[
	Q(f)\geq \frac{6m}{\rho^3}-C''\rho^{-4},
	\]
	which completes the proof of Theorem \ref{stab:fb}.

	From this point on, the proof that the family of free boundary CMC hemispheres comprises a foliation  follows from a simple variation of the standard argument. Indeed, if $\chi_0<\chi_1$ are the first two (unconstrained) eigenvalues of {{$\mathsf L_{S_{\rho,+}(\phi_\rho)}$, a spin-off of the analysis above leads to  
			\begin{equation}\label{est:invL}
				\chi_0=-\frac{2}{\rho^2}+\frac{10m}{\rho^3}+O(\rho^{-4}), \quad \chi_1\geq \frac{6m}{\rho^3}-C\rho^{-4}.
			\end{equation}
			In other words, $\mathsf L_{S_{\rho,+}(\phi_\rho)}:C_\bullet^{2,\alpha}(S_{\rho,+}(\phi_\rho))\to C^{\alpha}(S_{\rho,+}(\phi_\rho))$ is injective, where the bullet indicates that the boundary condition in (\ref{eingen:prob}) is imposed. Since it is  known that this Jacobi operator is Fredholm of index zero \cite[Section 2]{maximo2017free}, we see that it is surjective as well. In particular, there exists  $f_\rho\in C_\bullet^{2,\alpha}(S_{\rho,+}(\phi_\rho))$ such that $\mathsf L_{S_{\rho,+}(\phi_R)}f_\rho = 1$. 
			On the other hand, just like in the discussion after the proof of Proposition \ref{nonvan} above, we may realize $f_\rho$ as the variational function associated to a deformation of  $S_{\rho,+}^2(\phi_{\rho})$ by the graphical free boundary CMC hemispheres, now parameterized by their mean curvature $H\in (H_\rho-\epsilon,H_\rho+\epsilon)$, $\epsilon>0$. Since 
			the Nash-Moser scheme may be implemented to make sure
			that $f_\rho$ never vanishes, the standard argument using the inverse function theorem shows that this deformation actually provides a diffeomorphism of $(H_\rho-\epsilon,H_\rho+\epsilon)\times S_{\rho,+}^2(\phi_{\rho})$ onto a small neighborhood of $S_{\rho,+}^2(\phi_{\rho})$ in $M$. This proves the existence of the foliation and completes the proof of Theorem \ref{free:af:bd}.

			\section{Large relative isoperimetric hemispheres: the proof of Theorem \ref{iso:ext:em1}}\label{large:rel:iso}
			
			Here we follow \cite{eichmairlarge} closely  and present the proof of Theorem \ref{iso:ext:em1}.
			We start by briefly reviewing the argument leading to their main result, which  is based on three ingredients:
			\begin{itemize}
				\item An effective area comparison result for large volume, off-center regions  in Schwarz\-schild space, which 
				refines Bray's characterization of isop\-er\-imetric regions as being those enclosed by centered spheres (together with the minimal horizon) \cite{bray1997penrose,corvino2007isoperimetric}. This is then transplanted to an effective estimate for large volume, off-center regions in asymptotically Schwarz\-schild manifolds; see Proposition 3.3 and Theorem 3.4 in \cite{eichmairlarge}.
				\item A precise understanding of the behavior of minimizing sequence of regions attaining the corresponding isoperimetric profile, to the effect that they split as the disjoint union of  a (possibly empty) isoperimetric region (for the volume it encloses) that remains at a finite distance of a given point and a coordinate ball of radius $r\geq 0$ which slides away toward the asymptotic region. Moreover, if none of these regions degenerate (in particular, $r>0$) then the boundary of the isoperimetric region left behind has constant mean curvature $2/r$; see \cite[Proposition 4.2]{eichmairlarge}, which  relies on \cite[Theorem 2.1]{ritore2004existence}.
				\item  Existence of a foliation by CMC spheres filling out the asymptotic region as in \cite{huisken1996definition,ye1997foliation, huang2008center}.  
			\end{itemize}
			
			The first item above is used to make sure that the boundary of a sufficiently large isoperimetric region remains close to a centered coordinate sphere bounding the same volume in the sense that the scale invariant $C^2$-norm of the function describing
			such large isoperimetric surface as a normal graph over the 
			centered sphere tends to zero as the enclosed volume goes to
			infinity. Otherwise, by suitably scaling down the region one is able to check that it is off-center, hence not isoperimetric by the effective area estimate, a contradiction; see \cite[Theorem 5.1]{eichmairlarge}. With this information at hand, one sees from the second item above that for a large enclosed volume the worst case scenario takes place whenever the runaway ball does not degenerate, for in this case the isoperimetric region attaining this volume  splits as the disjoint union of two large balls each roughly with the same radius. Since this configuration is far from being isoperimetric, we get a contradiction. Thus, the runaway ball actually disappears and the isoperimetric region starts filling out the whole manifold as the enclosed volume diverges. Moreover, as its boundary remains close to a centered coordinate sphere, it has to coincide with a leaf of the foliation appearing in the third item.
			
			It turns out that all of these ingredients also work fine in our setting. Of course, the existence of the relevant foliation, in our case by free boundary CMC hemispheres, is the content of Theorem \ref{free:af:bd} above.  We now discuss the validity of the remaining ones. 
			
			First, it is clear that in half-Schwarzschild space, the region bounded by the minimal horizon $r=m/2$ and a coordinate sphere of radius $r>m/2$ is the only one attaining the  relative isoperimetric profile for the corresponding volume. Otherwise, after reflecting upon the totally geodesic boundary $x_3=0$ we obtain a region in ``boundaryless'' Schwarzschild space which is isoperimetric but differs from any of the symmetric regions realizing the corresponding isoperimetric profile as in Bray's result. Essentially the same argument yields an effective area comparison for off-center regions in ahS manifolds. 
			
			\begin{definition}\label{off:center}
				Let $(M,g,\Sigma)$ be as in Theorem \ref{iso:ext:em1}. Given  $(\tau,\eta)\in(1,+\infty)\times(0,1)$, a bounded Borel set $\Omega\subset M$ of finite relative perimeter is said to be $(\tau,\eta)$-off-center if:
				\begin{enumerate}
					\item there exists a large coordinate hemisphere $S^2_{r,+}$, $r>1$, whose enclosed region, say $M_r$, has the same volume as $\Omega$;
					\item $\mathcal H^2_g(\partial^*\Omega\backslash M_{\tau r})\geq \eta\mathcal H_2(S^2_{r,+})$. 
				\end{enumerate} 
				Here, $\mathcal H^2_g$ is Hausdorff measure with respect to $g$.
			\end{definition}
			
			The next proposition provides the analogue of the first item above to our setting.  
			
			\begin{proposition}\label{off:cent:est}
				Let $(M,g,\Sigma)$ be as in Theorem \ref{iso:ext:em1}. For every $(\tau,\eta)\in(1,+\infty)\times(0,1)$ there exists $V_0>0$ and $\Theta>0$ such that the following holds. Let $\Omega\subset M$ be  a bounded Borel set with finite relative perimeter whose volume is at least $V_0$ and  which is  $(\tau,\eta)$-off-center and further satisfies $\mathcal H^2_g(\partial^*\Omega)^{1/2}{\rm vol}_g(\Omega)^{-1/3}\leq \Theta$ and $\mathcal H^2_g(M_\sigma\cap \partial^*\Omega)\leq\Theta\sigma^2$ for all $\sigma\geq 1$. Then,
				\begin{equation}\label{effec:bound}
					\mathcal H^2_g(S^2_{r,+})+cr\leq \mathcal H^2_g(\partial^*\Omega), \quad c=c(m,\tau,\eta)>0.
				\end{equation}
			\end{proposition}    
			
			\begin{proof}
				First note that an effective bound similar to (\ref{effec:bound}) holds in case $\Omega$ is a subset of the half-Schwarzschild space. Indeed, upon reflecting this $\Omega$ across the totally geodesic boundary $x_3=0$ we obtain a region in the exact ``boundaryless'' Schwarzschild space to which \cite[Proposition 3.3]{eichmairlarge} applies. By halving the so obtained estimate, the sought for bound follows. As already emphasized in \cite{eichmairlarge}, this bound is robust enough to provide, via a suitable scaling argument, the effective area estimate (\ref{effec:bound}) for a  $(\tau,\eta)$-off-center region in a ahS manifold as in the theorem. The argument is virtually identical to the one appearing in the proof of \cite[Theorem 3.4]{eichmairlarge}, so it is omitted here.  	
			\end{proof}
			
			By a scaling argument as in the proof of \cite[Theorem 5.1]{eichmairlarge}, we check that isoperimetric surfaces remain close to a centered hemisphere as the volume diverges. We now take a divergent sequence of volumes $V_i\to +\infty$. Arguing as in \cite[Proposition 4.2]{eichmairlarge}, we see that there exits a fixed isoperimetric region $\Omega_i$ and a coordinate half-ball of radius $r_i\geq 0$  which is disjoint from $\Omega_i$ and contributes to the attained isoperimetric profile in the expected manner:
			\[
			{\rm vol}_3(\Omega_i)+\frac{2\pi r_i^3}{3}=V_i,\quad \mathcal H_g^2(\partial^*\Omega_i)+2\pi r_i^2=I_g(V_i).
			\] 
			This is our analogue of the volume splitting in the second item above. Moreover, if $r_i>0$ as $i\to+\infty$ then the mean curvature of $\partial^*\Omega_i$ is $2/r_i$, so the relative isoperimetric region associated to $V_i$     
			encompasses two disjoints half-balls each roughly with the same radius $r_i$. This contradiction shows that $r_i=0$ for all $i$ large enough. Thus, to each volume greater than some $V_0$ the corresponding isoperimetric region stays  at a finite distance from a given point on the manifold. Since we already known that this region centers around a large coordinate hemisphere, it certainly sweeps out the whole manifold as the volume diverges and its boundary necessarily coincides with a free boundary CMC hemisphere described in Theorem \ref{free:af:bd}; see Appendix \ref{uniq:stab}  in regard to this last point. This completes our sketch of the proof of Theorem \ref{iso:ext:em1}. 
			
			\begin{remark}\label{key}
				As already pointed out in \cite{eichmairlarge}, the existence of relative isoperimetric regions for sufficiently large enclosed volumes via the argument above only requires that $g$ is $C^0$-asymptotic to  half-Schwarzschild. The higher order asymptotics, the Regge-Teitelboim condition included, are only needed to make sure that a foliation exists as in Theorem \ref{free:af:bd}, so its leaves may be identified to the isoperimetric hemispheres.  
			\end{remark}
			
			\begin{remark}\label{rem:regul}
				The argument above  assumes the well-known fact that relative isoperimetric surfaces are sufficiently regular (indeed smooth) up to the boundary and hence are stable free boundary CMC surfaces as explained in Appendix \ref{var:setup}. This most desirable property is explicitly stated in \cite[Proposition 2.4]{ritore2004existence} and we refer to the discussion there for the pertaining sources. 
			\end{remark}

			\begin{remark}\label{asym:hyp}
				Very likely an analogue of Theorem \ref{free:af:bd} holds true for the class of asymptotically hyperbolic $3$-manifolds with a non-compact boundary introduced in \cite{almaraz2020mass}. This would extend a series of results  in the boundaryless case literature starting with \cite{rigger2004foliation,neves2009existence,neves2010existence,mazzeo2011constant}; see also \cite{nerz2018foliations} and the references therein for a recent account of the status of this line of research. In the same vein, it might also be possible to characterize the corresponding large relative isoperimetric regions in the line of the main result in \cite{chodosh2016large}, so as to extend Theorem \ref{iso:ext:em1} accordingly.  
			\end{remark}

			%%%%%%%%%%%%%%%%%%%%%%%%%%%%%%%%%%%%%%%%%%%%%%%%%%%%%%%%%%%%% Appendix A

			\appendix
			
			\section{The large scale isoperimetric deficits and the mass: the proofs of Theorems  \ref{isovk} and \ref{remwithbd}}\label{isop:mass:app}
			
			{{
					The arguments to prove Theorems \ref{isovk} and \ref{remwithbd} are simple variations on the computation appearing in \cite[Section 2]{fan2007large}, where a proof of Theorem \ref{huisfst} appears. This justifies the inclusion of a somewhat detailed account of their calculation in what follows. }}
			
			{{
					If $(M,g)$ is asymptotically flat as in Definition \ref{af:eman},}} we first observe that, 
			since $\partial r/\partial x_i=x_i/r$, we have 
			\begin{equation}\label{drform}
				\nabla r=g^{ij}\frac{x_i}{r}\frac{\partial}{\partial x_j}.
			\end{equation}
			and hence
			\begin{equation}\label{asymgradr}
				|\nabla r|^2=g^{ij}\frac{x_ix_j}{r^2}=1-e_{ij}\frac{x_ix_j}{r^2}+O(r^{-2\tau}). 
			\end{equation}
			Also, if  $\nu$ is the outward unit normal vector field to the coordinate $2$-sphere $S^2_r$ then  
			\begin{equation}\label{nu2}
				\nu=\frac{x}{r}+O(r^{-\tau}). 
			\end{equation}

			Let $dS^{2,\delta}_r=r^2dS^{2,\delta}_1$ be the area element of the Euclidean sphere of radius $r$. It follows that the area element of the corresponding coordinate sphere $S^2_r$ expands as
			\begin{equation}\label{areaexp}
				dS_r^2=\left(1+\frac{1}{2}h^{ij}e_{ij}+O(r^{-2\tau})\right)dS_r^{2,\delta}, 
			\end{equation} 
			where 
			\begin{equation}\label{indmetricexp}
				h_{ij}=g_{ij}-\nu_i\nu_j=\delta_{ij}-\frac{x_ix_j}{r^2}+O(r^{-\tau})
			\end{equation} 
			is the induced metric (extended to vanish in the radial direction). Thus, the  area of $S^{2}_r$ is 
			\begin{equation}\label{exparea}
				A(r)=4\pi r^2+\frac{1}{2}\int_{S_r^2}h^{ij}e_{ij}dS_r^{2,\delta}+O(r^{2-2\tau}). 
			\end{equation}
			From this we obtain
			\begin{eqnarray*}
				\frac{d}{dr}A(r) 
				& = & 8\pi r+\frac{1}{2}\int_{S_r^2}h^{ij}\frac{x_k}{r}e_{ij,k}dS_r^{2,\delta}+\frac{1}{r}\int_{S^2_r}h^{ij}e_{ij}dS_r^{2,\delta}+O(r^{1-2\tau}),
			\end{eqnarray*}
			where the comma means partial differentiation. Using (\ref{indmetricexp}) we get
			\begin{eqnarray*}
				\frac{d}{dr}A(r) & = & 8\pi r + \frac{1}{2}\int_{S_r^2}e_{ii,k}\frac{x_k}{r}dS_r^{2,\delta}-\frac{1}{2}\int_{S_r^2}e_{ij,k}\frac{x_ix_jx_k}{r^3}dS_r^{2,\delta}\\
				& & \quad +\frac{1}{r}\int_{S^2_r}h^{ij}e_{ij}dS_r^{2,\delta}+O(r^{1-2\tau}).
			\end{eqnarray*}
			
			We now work out the third term in the right-hand side. We first note that
			\begin{equation}\label{forderq}
				\frac{\partial}{\partial x_i}\frac{x_j}{r}=\frac{\delta_{ij}}{r}-\frac{x_ix_j}{r^3}. 
			\end{equation}
			We then compute:
			\begin{eqnarray*}
				\int_{S_r^2}\frac{\partial}{\partial x_k}\left(e_{ij}\frac{x_j}{r}\right)\frac{x_ix_k}{r^2}dS_r^{2,\delta} & = & \int_{S_r^2}e_{ij,k}\frac{x_ix_jx_k}{r^3}dS_r^{2,\delta}\\
				&& \quad +\int_{S_r^0}e_{ij}\left(\frac{\delta_{jk}}{r}-\frac{x_jx_k}{r^3}\right)\frac{x_ix_k}{r^2}dS_r^{2,\delta}\\
				& = & \int_{S_r^2}e_{ij,k}\frac{x_ix_jx_k}{r^3}dS_r^{2,\delta},
			\end{eqnarray*}
			so we have
			\begin{eqnarray*}
				\int_{S_r^2}e_{ij,k}\frac{x_ix_jx_k}{r^3}dS_r^{2,\delta}
				& = & 	\int_{S_r^2}\frac{\partial}{\partial x_k}\left(e_{ij}\frac{x_j}{r}\right)\frac{x_ix_k}{r^2}dS_r^{2,\delta}\\
				& = & \underbrace{\int_{S_r^2}\frac{\partial}{\partial x_i}\left(e_{ij}\frac{x_j}{r}\right)dS_r^{2,\delta}}_{(I)}\\
				& & \quad 	\underbrace{-
					\int_{S^2_r}\left(\delta_{ik}-\frac{x_ix_k}{r^2}\right)\frac{\partial}{\partial x_k}\left(e_{ij}\frac{x_j}{r}\right)dS_r^{2,\delta}}_{(II)}.
			\end{eqnarray*}
			Using (\ref{indmetricexp}) we have
			\[
			(I)=\int_{S_r^2}e_{ij,i}\frac{x_j}{r}dS_r^{2,\delta}+\frac{1}{r}\int_{S_r^2}h^{ij}e_{ij}dS_r^{2,\delta}+O(r^{1-2\tau}).
			\]
			Also, integration by parts together with (\ref{forderq}) gives
			\begin{equation}\label{intparts}
				(II)=-\int_{S^2_r}\frac{\partial}{\partial x_k}\left(\frac{x_ix_k}{r^2}\right)e_{ij}\frac{x_j}{r}dS_r^{2,\delta}=-2\int_{S^2_r}e_{ij}\frac{x_ix_j}{r^3}dS_r^{2,\delta},
			\end{equation}
			so that 
			\begin{eqnarray*}
				\int_{S_r^2}e_{ij,k}\frac{x_ix_jx_k}{r^3}dS_r^{2,\delta} & = & -2\int_{S^2_r}e_{ij}\frac{x_ix_j}{r^3}
				dS_r^{2,\delta}+\int_{S^2_r}e_{ij,i}\frac{x_j}{r}dS_r^{2,\delta}\\
				& & \quad +\frac{1}{r}\int_{S^2_r}h^{ij}e_{ij}
				dS_r^{2,\delta}+O(r^{1-2\tau}). 
			\end{eqnarray*}
			Thus,
			\begin{eqnarray}\label{areaderform}
				\frac{d}{dr}A(r)
				&=&{{8\pi r+\frac{1}{2}\int_{S^2_r}(e_{ii,j}-e_{ij,i})\frac{x_j}{r}dS_r^{2,\delta}\notag}}\\
				& &{{ +\int_{S^2_r}e_{ij}\frac{x_ix_j}{r^3}dS_r^{2,\delta}+\frac{1}{2r}\int_{S^2_r}h^{ij}e_{ij}dS_r^{2,\delta}+O(r^{1-2\tau})}}\notag\\
				&=&8\pi r-8\pi m+\int_{S^2_r}e_{ij}\frac{x_ix_j}{r^3}dS_r^{2,\delta}+\frac{1}{2r}\int_{S^2_r}h^{ij}e_{ij}dS_r^{2,\delta}+o(1).
			\end{eqnarray}
			Combining this with (\ref{exparea}), we get
			\begin{equation}\label{areaderform2}
				\frac{d}{dr}A(r)=\frac{A(r)}{r}+4\pi r-8\pi m+\int_{S^2_r}e_{ij}\frac{x_ix_j}{r^3}dS_r^{2,\delta}+o(1).
			\end{equation}
			
			We now look at the volume $V(r)$ enclosed by $S_r^2$. By the co-area formula, 
			(\ref{asymgradr}) and (\ref{areaexp}),
			\begin{eqnarray}\label{interres}
				\frac{1}{r}\frac{d}{dr}V(r)
				& = & 
				\frac{1}{r}\int_{S_r^2}|\nabla r|^{-1} dS^2_r\nonumber\\
				&= & 
				\frac{A(r)}{r}+\frac{1}{2}\int_{S_r^2}e_{ij}\frac{x_jx_j}{r^3}dS_r^{2,\delta}+o(1).
			\end{eqnarray}
			We may now eliminate the integral term in (\ref{areaderform2}) and (\ref{interres}). The result is 
			\[
			\frac{d}{dr}(rA(r))=4\pi r^2-8\pi mr+2\frac{d}{dr}V(r)+o(r).
			\]
			Integrating we obtain a formula relating the volume and area, namely,
			\begin{equation}\label{volarea}
				V(r)=\frac{1}{2}rA(r)-\frac{2\pi}{3}r^3+2\pi m r^2+o(r^2), 
			\end{equation}
			which gives
			\[
			J_r^{M;3,2}=r+\frac{4\pi r^2}{A(r)}\left(m-\frac{r}{3}\right)-\frac{2r}{3}\left(\frac{A(r)}{4\pi r^2}\right)^{\frac{1}{2}}+o(1).
			\]
			On the other hand, from (\ref{exparea}),
			\[
			\frac{A(r)}{4\pi r^2}=1+\mathcal I+O(r^{-2\tau}),\quad \mathcal I:=\frac{1}{8\pi r^2}\int_{S_r^2}h^{ij}e_{ij}dS_r^{2,\delta}=O(r^{-\tau})
			\]
			so that
			\begin{eqnarray*}
				J_r^{M;3,2}& = & r+\left(1-\mathcal I+O(r^{-2\tau})\right)\left(m-\frac{r}{3}\right)-\frac{2r}{3}\left(1+\frac{1}{2}\mathcal I+O(r^{-2\tau})\right)+o(1)\\
				& = & m+o(1),
			\end{eqnarray*}
			which gives the proof of Theorem \ref{huisfst}. 
			
			So far we have been following \cite{fan2007large} closely. We now explain how a little variation yields the proof of Theorem \ref{isovk}. 
			We will make use of the well-known expansion
			\begin{equation}\label{exphr}
				H=\frac{2}{r}+O(r^{-\tau-1}).
			\end{equation}
			Together with (\ref{areaexp}) this gives
			\begin{equation}\label{expk}
				\frac{M(r)}{8\pi r}=1+\mathcal I+O(r^{-2\tau}).  
			\end{equation}
			Also, by the first variation formula for the area,
			\begin{eqnarray*}
				\frac{d}{dr}A(r) & = &\int_{S_r^2}\left\langle\frac{\partial}{\partial r},H\nu\right\rangle dS^2_r\\
				& \stackrel{(\ref{asymgradr})+(\ref{exphr})}= & \int_{S_r^2}|\nabla r|^{-1}H dS^2_r\\
				& = & M(r)+\int_{S^2_r}e_{ij}\frac{x_ix_j}{r^3}dS^{2,\delta}_r+O(r^{-2\tau+1}),
			\end{eqnarray*}
			and combining this with (\ref{areaderform2}) we get
			\begin{equation}\label{relareak}
				\frac{1}{2}r^2M(r)=\frac{1}{2}r{A(r)}+2\pi r^3-4\pi mr^2+o(r^2).
			\end{equation}
			We now use (\ref{volarea}) to eliminate the area term. Solving for the volume we get
			\begin{equation}\label{relvolk}
				V(r)=\frac{1}{2}r^2M(r)-\frac{8\pi}{3}r^3+6\pi mr^2+o(r^2),
			\end{equation}
			so that, using (\ref{expk}),
			\begin{eqnarray*}
				J_r^{M;3,1} & = & \frac{2}{3}r+\frac{8\pi r}{M(r)}\left(m-\frac{4}{9}r\right)-\frac{2}{9}r\left(\frac{M(r)}{8\pi r}\right)^2+o(1)\\
				& = & \frac{2}{3}r+\left(1-\mathcal  I+O(r^{-2\tau})\right)\left(m-\frac{4}{9}r\right)-\frac{2}{9}r\left(1+2\mathcal  I+O(r^{-2\tau})\right)+o(1)\\
				& = & m + o(1),
			\end{eqnarray*}
			which finishes the proof of the first equality in (\ref{hkiso2}). As for the second one, note that  by (\ref{relareak}) and (\ref{expk}),   
			\begin{eqnarray*}
				J_r^{M;2,1} & = & r+\frac{4\pi r}{M(r)}\left(2m-r\right)-\frac{1}{16\pi}M(r)+o(1)\\
				& = & r+\frac{1}{2}\left(1-\mathcal I+O(r^{-2\tau})\right)\left(2m-r\right)-\frac{r}{2}\left(1+\mathcal I+O(r^{-2\tau})\right)+o(1)\\
				& = & m+o(1),
			\end{eqnarray*}
			which completes the proof of Theorem \ref{isovk}.
			
			We now present the proof of Theorem \ref{remwithbd}.
			We first observe that instead of (\ref{exparea}) we now have 
			\begin{equation}\label{expareabd}
				\mathcal A (r)=2\pi r^2+\frac{1}{2}\int_{S_{r,+}^2}h^{ij}e^+_{ij}dS_{r,+}^{2,\delta^+}+O(r^{2-2\tau}). 
			\end{equation} 
			Also, the integration by parts leading to (\ref{intparts}) now produces an extra term, so that $(II)$ gets replaced by
			\begin{equation}\label{intparts2}
				(II^+)=-2\int_{S^2_{r,+}}e^+_{ij}\frac{x_ix_j}{r^3}dS_{r,+}^{2,\delta^+}-\int_{ S^1_{r}}e^+_{kj}\frac{x_j}{r}\vartheta^k dS^{1,\delta^+}_r+O(r^{-2\tau+1}).
			\end{equation} 
			Thus, instead of (\ref{areaderform2}) we now have
			\begin{equation}\label{areaderformbd}
				\frac{d}{dr}\mathcal A(r)=\frac{\mathcal A(r)}{r}+2\pi r-8\pi \mathfrak m+\int_{S^2_{r,+}}e^+_{ij}\frac{x_ix_j}{r^3}dS_{r,+}^{2,\delta^+}+o(1).
			\end{equation} 
			Hence, proceeding exactly as before we now get
			\begin{equation}\label{volareabd}
				\mathcal V(r)=\frac{1}{2}r\mathcal A(r)-\frac{\pi}{3}r^3+2\pi \mathfrak m r^2+o(r^2), 
			\end{equation}
			which gives 
			\begin{eqnarray*}
				\mathcal J^{M;3,2}_r & = & \frac{r}{2}+\frac{2\pi r^2}{\mathcal A(r)}\left(\mathfrak m -\frac{r}{6}\right)-\frac{r}{3}\left(\frac{\mathcal A(r)}{2\pi r^2}\right)^{\frac{1}{2}}+o(1)\\
				& = &  \frac{r}{2}+\left(1-\widehat I+O(r^{-2\tau})\right)\left(\mathfrak m -\frac{r}{6}\right)-\frac{r}{3}\left(\left(1+\frac{1}{2}\widehat I+O(r^{-2\tau})\right)\right)+o(1)\\
				& = & \mathfrak m +o(1),
			\end{eqnarray*}
			where 
			\[
			\widehat I=\frac{1}{4\pi r^2}\int_{S^2_{r,+}}h^{ij}e^+_{ij}dS^2_{r,+}=O(r^{-\tau}).
			\] This completes the proof of Theorem \ref{remwithbd}.

			%%%%%%%%%%%%%%%%%%%%%%%%%%%%%%%%%%%%%%%%%%%%%%%%%%%%%%%%%%%%% Appendix B

			\section{The variational setup}\label{var:setup}

			Here we address the variational issues needed in the bulk of the paper. Our aim is twofold. First,  we review the well-known  variational theory of free boundary constant mean curvature surfaces \cite{ros1995stability,ros1997stability} {in a way that is convenient for our purposes}. Next, we  discuss the much less known variational theory  of closed surfaces which are critical for the total mean curvature  functional under a volume preserving constraint and develop the corresponding stability theory. We remark that the variational theory associated to curvature integrals involving elementary symmetric functions of the principal curvatures (quermassintegrals) of hypersurfaces in {\em space forms} is a well established subject; see  \cite{barbosa1997stability} and the references therein.

			We start by considering a one-parameter family  of compact, embedded surfaces $t\in (-\varepsilon,\varepsilon)\mapsto S_t$ in an arbitrary Riemannian manifold $(M^3,g)$ evolving as 
			\begin{equation}\label{var:pres}
				\frac{\partial x_t }{\partial t}=Y_t,
			\end{equation} 
			where $x_t$ is the smooth map defining the embedding and $Y_t$ is a vector field along $S_t$, a (not necessarily normal) section of $TM$ restricted to $S_t$.  As usual, if $\nu_t$ is the unit normal vector field along $S_t$, let $W=\nabla\nu_t$ be the shape operator of $S_t$, so the corresponding principal curvatures (the eigenvalues of $W$) are $\kappa_1$ and $\kappa_2$. Thus, the mean curvature is $H=\kappa_1+\kappa_2$ and the Gauss-Kronecker curvature is $K=\kappa_1\kappa_2$. For later reference, {we recall that}
			\[
			\widetilde K=K-\frac{1}{2}{\rm Ric}_g(\nu,\nu)
			\]
			is the modified Gauss-Kronecker curvature.
			
			A well-known computation gives 
			\begin{equation}\label{first:var:area2:0}
				\frac{d}{dt}A(t)|_{t=0}=\int_{ S} {\rm div}_S YdS,
			\end{equation}
			where $A(t)$ is the area of $S_t$ and  we agree to drop the subscript $t$ upon evaluation at $t=0$. 
			Next we decompose $Y_t$ into its normal and tangential components: 
			\begin{equation}\label{var:pres:dec}
				Y_t=f_t\nu_t+Y_t^\top, \quad f_t=\langle Y_t,\nu_t\rangle.
			\end{equation}
			Thus, if we assume further that $S_t$ carries a boundary $\partial S_t$, 
			\begin{eqnarray}\label{first:var:area}
				\frac{d}{dt}A(t)|_{t=0} & = & \int_SfHdS+\int_S{\rm div}_SY^\top dS\nonumber\\
				& = &  \int_SfHdS +\int_{\partial S}\langle Y,\mu\rangle d\partial S,
			\end{eqnarray}
			where $\mu$ is the outward unit normal vector field along $\partial S$ and we used that $H={\rm div}_S\nu$. 
			
			Let us assume now that $M$  also  carries a  boundary, say $\Sigma$, with the variation being {\em admissible} in the sense that  $\partial S_t\subset\Sigma$. It follows   
			that $S=S_0$ is critical for the area under such variations satisfying the volume preserving condition
			\begin{equation}\label{pres:vol}
				\int_SfdS=0
			\end{equation}
			if and only if the mean curvature is constant and $S$ meets $\Sigma$ orthogonally along $\partial S$. We then say that $S$ is a free boundary constant mean curvature (CMC) surface.

			We now recall the corresponding notion of stability. Assuming that $S=S_0$ is a free boundary CMC as above,  a well-known computation \cite{ros1997stability} gives
			the second variational formula for the area: 
			\begin{equation}\label{secvar:areab}
				\frac{d^2 A}{dt^2}|_{t=0}=\int_{S}f\mathscr L_SfdS+\int_{\partial S}  f\left(\frac{\partial f}{\partial \mu}-\kappa f\right) d\partial S,
			\end{equation}
			where
			\begin{equation}\label{stab:op:0}
				\mathscr L_S=-\Delta_S -\left(|W|^2+{\rm Ric}_{g}(\nu,\nu)\right),
			\end{equation} 
			$\kappa=\langle \nu,\mathcal W \nu\rangle$ and $\mathcal W=\nabla \eta$ is the shape operator of the embedding $\Sigma\hookrightarrow M$. Here, $\eta$ is the outward unit normal vector to $M$ along $\Sigma$.

			Recall that $S=S_0$ is strictly stable (as a free boundary CMC surface) if the right-hand side of (\ref{secvar:areab}) is positive for any $f\neq 0$ satisfying (\ref{pres:vol}).  
			Accordingly, we define
			\[
			\mathcal F(S)=\left\{f\in H^1(S);\int_SfdS=0\right\}. 
			\] 
			
			\begin{proposition}\label{stab:cmcfree}
				A free boundary CMC surface $S$ as above is strictly stable if and only if the first eigenvalue $\lambda_{\mathscr L_S}$ of the eigenvalue problem  
				\[
				\left\{
				\begin{array}{ll}
					\mathscr L_Sf=\lambda f & {\rm in}\,\,\, S, \\
					\frac{\partial f}{\partial \mu}=\kappa f & {\rm on}\,\,\, \partial S,
				\end{array}
				\right.
				\]	
				is positive, where $f\in \mathcal F (S)$. Equivalently, for any $0\neq f\in \mathcal F(S)$,
				\[
				\int_S\left(|\nabla_Sf|^2-\left(|W|^2+{\rm Ric}(\nu,\nu)\right)f^2\right)dS-\int_{\partial S}\kappa f^2d\partial S> 0. 
				\]
			\end{proposition}

			We now turn to the variational theory of the total mean curvature functional $\int_S H dS$. Here we assume that $S_t$ is closed ($\partial S=\emptyset$) and the variation is normal ($Y=f\nu$).  
			A simple computation shows that the shape operator evolves as 
			\begin{equation}\label{varshape}
				\frac{\partial W}{\partial t}=-\nabla_S^2 f-(W^2+{\rm Riem}_{g}^\nu)f,
			\end{equation}
			where $\nabla_S^2$, the Hessian of $f$, is viewed as a $(1,1)$-tensor, ${\rm Riem}^\nu_g(\cdot)={\rm Riem}_g(\cdot,\nu)\nu$ and ${{W^2=W\circ W}}$. 
			%Note that (\ref{evolmean}) follows from this by tracing with respect to $h$.

			\begin{proposition}\label{k:vol:var}
				In a Riemannian $3$-manifold $(M,g)$ as above, a closed surface extremizes the total mean curvature under volume (respectively, area) preserving variations if and only if $\widetilde K={\rm const}$ (respectively, $\widetilde K=\gamma H$, where $\gamma$ is a  constant). 
			\end{proposition}
			
			\begin{proof}
				From $\partial dS_t/\partial t=fHdS_t$, the fact that the mean curvature evolves as 
				\begin{equation}\label{evolmean}
					\frac{\partial H}{\partial t}=\mathscr L_Sf, 
				\end{equation}
				and the algebraic identity $|W|^2=H^2-2K$, we immediately see that 
				\[
				\frac{\partial}{\partial t}\int_{S_t}H dS_t|_{t=0}=2\int_S \widetilde K fdS_t, 
				\]
				which proves the first statement. As for the second one, just combine the computation above with (\ref{first:var:area}) and take into account that $\partial S=\emptyset$.
			\end{proof}

			In order to discuss the stability of this variational problem, we now compute the variation of $\widetilde K$. First, from $\partial\nu/\partial t=-\nabla_Sf$,  
			\begin{equation}
				\label{varricterm}
				\frac{\partial}{\partial t}{\rm Ric}_g(\nu,\nu)=f(\nabla_\nu{\rm Ric}_g)(\nu,\nu)-2{\rm Ric}_g({\nabla_Sf},\nu).
			\end{equation}
			As for the variation of $K$, we first recall the well-known formula 
			\[
			\frac{\partial }{\partial t}K={\rm tr}_S\left(\Pi \frac{\partial}{\partial t}W\right), 
			\]
			where $\Pi=HI-W$ is the Newton tensor \cite{rosenberg1993hypersurfaces}. Using (\ref{varshape}) we then get
			\begin{equation}\label{var:K}
				\frac{\partial}{\partial t}K=-{\rm tr}_S(\Pi \nabla_S^2f)-f{\rm tr}_S(\Pi\, W^2)-f{\rm tr}_S(\Pi\,{\rm Riem}_g^\nu).
			\end{equation}
			To proceed  we choose an orthonormal frame $e_A$, $A=1,2$, tangent to $S$ with $(\nabla_S)_{e_A}e_B=0$ at the given point. We compute
			\begin{eqnarray*}
				{\rm tr}_S(\Pi \nabla_S^2f) & = & 	\Pi^{AB}\langle (\nabla_S)_{e_A}\nabla_Sf,e_B\rangle\\
				& = & \Pi^{AB}e_A\langle \nabla_Sf,e_B\rangle-\Pi^{AB}\langle \nabla_Sf,(\nabla_S)_{e_A}e_B\rangle\\
				& = & (\Pi^{AB}\nabla_Sf^B)_{;A}-\Pi^{AB}_{\,\,\,\,\,;A}\nabla_Sf^B,
			\end{eqnarray*}
			where the semicolon denotes covariant derivation.
			By Codazzi equations, recalling that $h=g|_S$,
			\begin{eqnarray*}
				\Pi^{AB}_{\,\,\,\,\,;A}	& = & (Hh^{AB})_{;A}-W^{AB}_{\,\,\,\,\,;A}\\
				& = & (e_AH)h^{AB}-W^{AA}_{\,\,\,\,\,;B}-{{\langle R(e_A,e_B)\nu,e_A\rangle}}\\
				& = & e_BH-e_BH-{\rm Ric}_{g}(\nu,e_B)\\
				& = & {{-{\rm Ric}_{g}(\nu,e_B)}},
			\end{eqnarray*}
			so that
			\[
			{\rm tr}_S(\Pi \nabla_S^2f)={\rm div}_S(\Pi \nabla_Sf)+{\rm Ric}_{g}(\nu,\nabla_Sf). 
			\]
			Thus, from (\ref{var:K}) and the algebraic identity ${\rm tr}_S(\Pi W^2)=HK$, {which holds in dimension $3$}, 
			\[
			\frac{\partial}{\partial t}K=-\Lambda_S f-{\rm Ric}_{g}(\nu,\nabla_Sf)-fHK-f{\rm tr}_S(\Pi{\rm Riem}_g^\nu),
			\]
			where 
			\begin{equation}\label{box:div}
				\Lambda_S f={\rm div}_S(\Pi \nabla_Sf).
			\end{equation}
			Together with (\ref{varricterm}) this finally gives
			\begin{equation}\label{varsig2tilde}
				\frac{\partial}{\partial t}\widetilde K=L_S f,
			\end{equation}
			where 
			\begin{equation}\label{stab:op:gk}
				L_S =-\Lambda_S {{-HK-\frac{1}{2}(\nabla_\nu{\rm Ric}_g)(\nu,\nu)-{\rm tr}_S(\Pi{\rm Riem}_g^\nu)}}
			\end{equation}
			is the corresponding Jacobi operator. We note that 
			\begin{equation}\label{self:adj}
				-\int_Sf\Lambda_S\tilde fdS=\int_S\langle\Pi\nabla_S f,\nabla_S \tilde f\rangle dS, 
			\end{equation}
			for any functions $f$ and $\tilde f$. 
			In particular, 
			$L_S$ is always self-adjoint. Moreover, it is easy to check that this operator is  elliptic whenever $\Pi$ is positive definite.

			We now consider a surface $S\subset M$ satisfying $\widetilde K={\rm const.}$ and with the property that $\Pi$ is positive definite everywhere. We then say that $S$ is strictly stable if 
			\[
			\frac{d^2}{dt^2}\int_{S}HdS|_{t=0}> 0,
			\]
			for any normal variation as in (\ref{var:pres}) with $f\neq 0$. 
			As before let us set 
			\[
			\mathcal G(S)=\left\{f\in H^1(S);\int_S fdS=0\right\}.
			\]

			\begin{proposition}\label{stable:cond}
				$S$ is strictly stable if and only if 
				\[
				\int_S\left( \langle\Pi\nabla_S f,\nabla_S f\rangle{{-f^2\left(HK+\frac{1}{2}(\nabla_\nu{\rm Ric}_g)(\nu,\nu)+{\rm tr}_S(\Pi{\rm Riem}_g^\nu)\right)}} \right)dS> 0,
				\]
				for any $0\neq f\in\mathcal G(S)$. Equivalently, the first eigenvalue $\lambda_{L_S}$ of the eigenvalue problem
				\[
				L_Sf=\lambda f, \quad f\in \mathcal G (S),
				\]
				is positive. 	 
			\end{proposition}

			%%%%%%%%%%%%%%%%%%%%%%%%%%%%%%%%%%%%%%%%%%%%%%%%%%%%%%%%%%%%% Appendix C

			\section{The Gauss-Kronecker curvature in terms of the mean curvature}\label{conf:field}
			
			In this section we prove Proposition \ref{k:func:h}. Thus we aim to prove the identity (\ref{K:func:h:f}) which expresses the Gauss-Kronecker curvature in terms of the mean curvature up to terms decaying fast enough at infinity. Our starting point is the fact that the radial vector field
			\[
			X=(x_i-a_i)\frac{\partial}{\partial x_i}
			\]
			is conformal with respect to the Euclidean metric, i.e., $\mathcal L_X\delta=2\delta$, where $\mathcal L$ is the Lie derivative. 
			From this we see that $X$ is also conformal with respect to the  metric $f_{m,c}^{\gamma_1, \gamma_2}\delta$ where 
			$$f_{m,c}^{\gamma_1, \gamma_2}=1+\frac{2m}{r}+\gamma_1\frac{c\cdot x}{r^3}+\gamma_2\frac{1}{r^2}$$
			for some $\gamma_1, \gamma_2\in \mathbb R$ and $c\in\mathbb R^3$. 
			Indeed,
			there holds $\mathcal L_X(f_{m,c}^{\gamma_1, \gamma_2}\delta)=2\xi f_{m,c}^{\gamma_1, \gamma_2}\delta$, with	
			\begin{eqnarray}\label{conf:gm1}
				\xi(x) & = & 
				\frac{1}{f_{m,c}^{\gamma_1, \gamma_2}}\left(f_{m,c}^{\gamma_1, \gamma_2}+\frac{1}{2}\partial_kf_{m,c}^{\gamma_1, \gamma_2}(x_k-a_k)\right)\\
				& = &
				1-\frac{m}{r}+\frac{2m^2-\gamma_2}{r^2}+{{\frac{m}{r^3} x\cdot a }}-{{\frac{\gamma_1}{r^3} x\cdot c }}+O(r^{-3})\nonumber\\
				& = & 1-\frac{m}{\rho}+\frac{2m^2-\gamma_2}{\rho^2}+{{\frac{2m}{\rho^3} x\cdot a }}-{{\frac{\gamma_1}{\rho^3} x\cdot c }}+O(\rho^{-3}),\notag 
			\end{eqnarray}
			where in the last step we used that 
			\begin{equation}\label{ajust:r:rho}
				r^k=\rho^k+k\frac{(x-a)\cdot a}{\rho^{2-k}}+O(\rho^{-2+k}), \quad k\in\mathbb R.
			\end{equation}
			
			Let us consider an aS metric of the form $g=f_{m,c}^{\gamma_1, \gamma_2}\delta+p$, where $p=O(r^{-2-\epsilon})$, which satisfies \eqref{asym-req} with $\epsilon\geq 0$.
			\begin{proposition}\label{alm:conf}
				The vector field $X$ is almost conformal with respect to $g$ in the sense that 
				\begin{equation}\label{alm:conf2}
					\mathcal L_Xg=2\xi g+B,  \qquad \text{where}\qquad B=O({\rho}^{-2-\epsilon}).
				\end{equation}
			\end{proposition}
			
			\begin{proof}
				A direct computation shows that (\ref{alm:conf2}) holds with $B=\mathcal L_Xp-2\xi p$. Note however that
				\[
				(\mathcal L_Xp)_{jk}=X^i\nabla_ip_{jk}+p_{ik}\nabla_jX^i+p_{ij}\nabla_kX^i,
				\]
				and the result follows given that $X=O(r)$. 
			\end{proof}
			
			We now take $\{e_1,e_2\}$ to be a local  orthonormal frame on $S^2_\rho(a)$ and $\nu$ its outward unit normal vector.  
			Recall that $W=\nabla\nu$ is the shape operator of $S^2_{\rho}(a)$ and $\Pi=HI-W$ denotes its Newton tensor.
			If $X^{\top}=X-\langle X,\nu\rangle\nu$ is the tangential component of $X$, then
			$$
			\mathcal L_{X^\top}g(\Pi e_A,e_A)=L_{X}g(\Pi e_A,e_A)-2\langle X,\nu\rangle W(\Pi e_A,e_A),
			$$ 
			and we obtain from (\ref{alm:conf2}) that 
			\[
			\langle\nabla_{\Pi e_A}X^\top,e_A\rangle+\langle\nabla_{e_A}X^\top ,\Pi e_A\rangle=2\xi\langle \Pi e_A,e_A\rangle-2\langle X,\nu\rangle\langle W\Pi e_A,e_A\rangle+B(\Pi e_A,e_A).
			\]
			Since 
			\[
			\langle\nabla_{e_A}X^\top ,\Pi e_A\rangle=\langle e_A,\nabla_{\Pi e_A}X^\top\rangle, 
			\]
			which is easily verified if we take the frame to be principal with respect to the shape operator $W$,
			this simplifies to 
			\[
			\langle\nabla_{e_A}X^\top ,\Pi e_A\rangle=\xi\langle \Pi e_A,e_A\rangle-\langle X,\nu\rangle\langle W\Pi e_A,e_A\rangle+\frac{1}{2}B(\Pi e_A,e_A).
			\]
			Thus, summing over $A$ and using that $H_{a,\rho}^2-|W|^2=2K_{a,\rho}$, we obtain
			\begin{equation}\label{esp:comp}
				\sum_A \langle\nabla_{e_A}X^\top ,\Pi e_A\rangle=\xi H_{a,\rho}-2\langle X,\nu\rangle K_{a,\rho}+\frac{1}{2}\sum_A B(\Pi e_A,e_A).
			\end{equation}
			In order to make use of this identity, which first appeared in \cite{alias2003integral}, we need to determine the asymptotics of $X^\top$.

			\begin{proposition}\label{xtang}
				One has $X^\top={O(\rho^{-1-\epsilon})}$. 
			\end{proposition}
			
			\begin{proof}
				Recalling that $\mathfrak r=(x-a)/\rho$, so that $X=\rho\mathfrak r_i\partial/\partial x_i$, one computes 
				\begin{equation}\label{normal:unit}
					\nu=\big(f_{m,c}^{\gamma_1, \gamma_2}\big)^{-1/2}\mathfrak r_i\frac{\partial}{\partial x_i}+O(\rho^{-2-\epsilon}),
				\end{equation}
				so that 
				\begin{equation}\label{xi:nu}
					\langle X,\nu\rangle=\big(f_{m,c}^{\gamma_1, \gamma_2}\big)^{1/2}\rho+O(\rho^{-1-\epsilon})
				\end{equation}
				Thus, 
				\[
				\langle X,\nu\rangle\nu  = \rho\,\mathfrak r_i\frac{\partial}{\partial x_i}+{O(\rho^{-1-\epsilon})},
				\]
				and the result follows. 
			\end{proof}

			We now observe that  by (\ref{mc:huang:f}) we may rewrite (\ref{exp:Pi}) as 
			\[
			\Pi=\frac{1}{2}H_{\rho,a}I+O(\rho^{-3}),
			\]
			so that 
			\[
			\frac{1}{2}\sum_AB(\Pi e_A,e_A)=\frac{1}{4}H_{\rho,a}{\rm tr}_{S^2_\rho(a)}B+O(\rho^{-5}), 
			\]
			where we used that ${{B=O(\rho^{-2})}}$.
			Also, the left-hand side of (\ref{esp:comp}) may be treated similarly. 
			Indeed, by Proposition \ref{xtang},
			\begin{eqnarray*}
				\sum_A \langle\nabla_{e_A}X^\top ,\Pi e_A\rangle 
				%& = & \left(\frac{1}{2}H_{\rho,a}+O(\rho^{-3})\right){\rm div}_{S^2_\rho(a)}X^\top\\
				& = & \frac{1}{2}H_{\rho,a}{\rm div}_{S^2_\rho(a)}X^\top+O(\rho^{-5}).
			\end{eqnarray*}

			Putting all the pieces of our computation together and using \eqref{xi:nu} we get 
			\begin{eqnarray*}
				2K_{a,\rho} & = &  \left(\frac{\xi}{\langle X,\nu\rangle}+\frac{1}{\langle X,\nu\rangle}\left(\frac{1}{4}{\rm tr}_{S^2_\rho(a)}B-\frac{1}{2}{\rm div}_{S^2_\rho(a)}X^\top\right)\right)H_{a,\rho}+O(\rho^{-6})\\
				& = & \left(\frac{\xi}{\langle X,\nu\rangle}+{ O(\rho^{-3-\epsilon})}\right)H_{a,\rho}+O(\rho^{-6}).
			\end{eqnarray*}
			The proof of Proposition \ref{k:func:h} is completed if we note that by (\ref{conf:gm1}) and (\ref{xi:nu}),
			\begin{eqnarray*}
				\frac{\rho\,\xi}{\langle X,\nu\rangle}
				& = &
				\xi\left(\big(f_{m,c}^{\gamma_1, \gamma_2}\big)^{-1/2}+O(\rho^{-2-\epsilon})\right)\\
				& = &
				\left(1-\frac{m}{\rho}+\frac{2m^2-\gamma_2}{\rho^2}+\frac{2m}{\rho^3}x\cdot a-\frac{\gamma_1}{\rho^3}x\cdot c+ O(\rho^{-3})\right)\\
				&  &\cdot\left(1-\frac{m}{\rho}+\frac{3m^2-\gamma_2}{2\rho^2}+\frac{m}{\rho^3}x\cdot a-\frac{\gamma_1}{2\rho^3}x\cdot c+ O(\rho^{-2-\epsilon})\right)\\
				& = & 1-\frac{2m}{\rho}+\frac{9m^2-3\gamma_2}{2\rho^2}+\frac{3m}{\rho^3}x\cdot a-\frac{3\gamma_1}{2\rho^3}x\cdot c+ O(\rho^{-2-\epsilon}).
			\end{eqnarray*}
			
			\section{The proof of Proposition \ref{mc:huang:mean:b}}\label{dens:res}
			
			Here we indicate how the argument in \cite[Appendix F]{eichmair2013unique} may be used to prove Proposition \ref{mc:huang:mean:b}. In fact, this method allows us to approach the problem in the category of manifolds considered in Definition \ref{regge:af:bd}.
			
			\begin{proposition}\label{regge:ads+}
				If $(M,g)$ is an asymptotically flat  $3$-manifold with a non-compact boundary satisfying the RT condition then 
				\begin{equation}\label{regge:ads+:eq}
					\int_{S^2_{\rho,+}(b)}\left(x_\alpha-b_\alpha\right)\left(H_{\rho,+,b}-\frac{2}{\rho}\right)dS^{2,\delta^+}_{\rho,+}(b)=8\pi\mathfrak m \left( b_\alpha-C_\alpha^+\right)+O(\rho^{-\tau}), \quad \alpha=1,2.
				\end{equation}
			\end{proposition}
			
			\begin{corollary}\label{cm:alt} There holds 
				\[
				\mathcal C_\alpha^+=-\lim_{\rho\to+\infty}\frac{1}{8\pi\mathfrak m}	\int_{S^2_{\rho,+}(\vec 0)}x_\alpha H_{\rho,+,\vec 0}dS^{2,\delta^+}_{\rho,+}(\vec 0).
				\]	
			\end{corollary}
			
			The key ingredient in the proof is an integral identity derived from the fact that $S^2_{\rho,+}(b)$ is a free boundary CMC surface with mean curvature $2/\rho$ with respect to the metric $\delta^+$. 
			
			\begin{proposition}\label{int:part:cmc}
				There holds
				\begin{eqnarray*}
					\frac{1}{2}\int_{S^2_{\rho,+}(b)}\left(x_\alpha-b_\alpha\right)e^+_{ij,k}\mathfrak r_i\mathfrak r_j\mathfrak r_kdS^{2,\delta^+}_{\rho,+}(b) & = & \int_{S^2_{\rho,+}(b)}\left(x_\alpha-b_\alpha\right)\left(\frac{1}{2}e^+_{ij,k}\mathfrak r_j-2e^+_{ij}\frac{\mathfrak r_i\mathfrak r_j}{\rho}\right)dS^{2,\delta^+}_{\rho,+}(b)\\
					& & \quad + \frac{1}{2}\int_{S^2_{\rho,+}(b)}\left(e^+_{ii}\mathfrak r_\alpha+e^+_{i\alpha}\mathfrak r_i\right) dS^{2,\delta^+}_{\rho,+}(b)\\
					& & \qquad -\frac{1}{2}\int_{ S^{1}_{\rho}(b)}\left(x_\alpha-b_\alpha\right)e^+_{i3}\mathfrak r_i dS^{1,\delta^+}_{\rho}(b), 
				\end{eqnarray*}
				where $S^{1}_{\rho}(b)=\partial S^2_{\rho,+}(b)$ and recall that $\mathfrak r=(x-b)/\rho$.
			\end{proposition}
			
			\begin{proof}
				Apply the identity that follows from equating the right-hand sides of (\ref{first:var:area2:0}) and (\ref{first:var:area}) with 
				$\mu=-\partial/\partial x_3$ to the vector field $Y_{(\alpha)}=(x_\alpha-b_\alpha) e^+_{ij}\mathfrak r_i\partial/\partial x_j$ by taking into account that 
				\[
				{\rm div}_{S^2_{\rho,+}(b)}\,Y_{\alpha}= e^{+}_{i\alpha}\mathfrak r_i+\left(x_\alpha-b_\alpha\right)
				\left(\frac{e^+_{ii}}{\rho}-2\frac{e^+_{ij}}{\rho}\mathfrak r_i\mathfrak r_j+e^+_{ij,j}\mathfrak r_i-e^{+}_{ij,k}\mathfrak r_i\mathfrak r_j\mathfrak r_k\right).
				\]
			\end{proof}
			
			We now recall the expansion 
			\[
			H_{\rho,+,b}-\frac{2}{\rho}=\frac{1}{2}e^+_{ij,k}\mathfrak r_i\mathfrak r_j\mathfrak r_k+2e^+_{ij}\frac{\mathfrak r_i\mathfrak r_j}{\rho}-e^{+}_{ij,i}\mathfrak r_j+\frac{1}{2}e^+_{ii,j}\mathfrak r_j-\frac{e^+_{ii}}{\rho}+E,
			\] 
			where the remainder satisfies {$E={O}(\rho^{-1-2\tau})$ and
				$E^{({-1}')}=O(\rho^{-2-2\tau})$; see \cite[Lemma 2.1]{huang2009center}.} 
			This reduces to (\ref{mc:huang:f:b}) if we take $e^+=2mr^{-1}\delta^++{O}(r^{-2})$, which provides the link between Propositions \ref{regge:ads+} and \ref{mc:huang:mean:b}. 
			It follows that 
			\begin{eqnarray*}
				\int_{S^2_{\rho,+}(b)}\left(x_\alpha-b_\alpha\right)\left(H_{\rho,+,b}-\frac{2}{\rho}\right)dS^{2,\delta^+}_{\rho,+}(b) & = &
				-\frac{1}{2} \int_{S^2_{\rho,+}(b)}\left(x_\alpha-b_\alpha\right)\left(e^+_{ij,i}-e^+_{ii,j}\right)\mathfrak r_j dS^{2,\delta^+}_{\rho,+}(b)	\\
				& & \quad +\frac{1}{2}	\int_{S^2_{\rho,+}(b)}\left(e^+_{i\alpha}\mathfrak r_i-e^+_{ii}\mathfrak r_\alpha\right)dS^{2,\delta^+}_{\rho,+}(b)\\
				& & \qquad -\frac{1}{2}\int_{ S^1_{\rho}(b)}\left(x_\alpha-b_\alpha\right)e^+_{i3}\mathfrak r_i d S^{1,\delta^+}_{\rho}(b)+ O({\color{blue}{\rho}}^{-\tau}),  
			\end{eqnarray*}
			where Proposition \ref{int:part:cmc} has been used to make sure that only those terms which are linear in $\mathfrak r$ survive in the right-hand side. 
			We now observe that under the decay assumptions (including Regge-Teitelboim) the integrals 
			\[
			\int_{S^2_{\rho,+}(b)}x_\alpha \left(e^+_{ij,i}-e^+_{ii,j}\right)\frac{b_{j}}{\rho}dS^{2,\delta^+}_{\rho,+}(b), \quad 
			\int_{S^2_{\rho,+}(b)} \left(e^+_{ij,i}-e^+_{ii,j}\right)\frac{b_{j}}{\rho}dS^{2,\delta^+}_{\rho,+}(b),
			\]
			and 
			\[
			\int_{S^2_{\rho,+}(b)}\left(e^+_{i\alpha}\frac{b_i}{\rho}-e^+_{ii}\frac{b_\alpha}{\rho}\right)dS^{2,\delta^+}_{\rho,+}(b)
			\]
			are $O(\rho^{-\tau})$, the same happening to the boundary integrals 
			\[
			\frac{b_{\alpha}}{\rho}\int_{S^1_{\rho}(b)}x_\alpha e^+_{i3}d S^{1,\delta^+}_{\rho}(b), \quad
			\frac{b_{\alpha}b_i}{\rho}\int_{S^1_{\rho}(b)} e^+_{i3}dS^{1,\delta^+}_{\rho}(b).
			\]
			Thus, we end up with 
			\begin{eqnarray*}
				\int_{S^2_{\rho,+}(b)}\left(x_\alpha-b_\alpha\right)\left(H_{\rho,+,b}-\frac{2}{\rho}\right)dS^{2,\delta^+}_{\rho,+}(b) & = & -\frac{1}{2}
				\int_{S^2_{\rho,+}(b)}x_\alpha \left(e^+_{ij,i}-e^+_{ii,j}\right)\frac{x_{j}}{\rho}dS^{2,\delta^+}_{\rho,+}(b)\\
				& & +\frac{1}{2}\int_{S^2_{\rho,+}(b)}\left(e^+_{i\alpha}\frac{x_i}{\rho}-e^+_{ii}\frac{x_\alpha}{\rho}\right)dS^{2,\delta^+}_{\rho,+}(b)\\
				& & -\frac{1}{2}\int_{S^1_{\rho}(b)}x_\alpha e^+_{i3}\frac{x_{i}}{\rho}dS^{1,\delta^+}_{\rho}(b)\\
				& & 	+\frac{1}{2} b_\alpha
				\int_{S^2_{\rho,+}(b)} \left(e^+_{ij,i}-e^+_{ii,j}\right)\frac{x_{j}}{\rho}dS^{2,\delta^+}_{\rho,+}(b)\\
				& &	+\frac{1}{2}b_\alpha\int_{S^1_{\rho}(b)} e^+_{i3}\frac{x_{i}}{\rho}d S^{1,\delta^+}_{\rho}(b)+O(\rho^{-\tau}). 
			\end{eqnarray*}
			Comparing the right-hand side of the above with the definitions of $\mathfrak m$ and $\mathcal C^+$, the proof of Proposition \ref{regge:ads+}, and hence of Proposition \ref{mc:huang:mean:b}, follows.

			\begin{remark}\label{various:inc}
				The upshot of Corollary \ref{cm:alt} is another expression for the center of mass $\mathcal C^+$, besides (\ref{cm:bd}), derived from Hamiltonian methods, and the isoperimetric one appearing in Theorems \ref{free:af:bd} and \ref{iso:ext:em1}. Another rendition of this invariant comes from \cite[Theorem 2.4]{de2019mass}, this time in terms of certain asymptotic flux integrals involving the Einstein tensor of the metric in the interior and the Newton tensor along the boundary; see also \cite{chai2018two}. 
				It is remarkable indeed that this kind of invariant admits so many distinct manifestations.    
			\end{remark}
			
			\section{The uniqueness of the free boundary CMC hemispheres}\label{uniq:stab}
			
			The very last piece of the argument  leading to Theorem \ref{iso:ext:em1} uses the appropriate uniqueness of the free boundary CMC hemispheres  in Theorem \ref{free:af:bd}. Here we justify this step by following the reasoning in \cite[Section 4]{huisken1996definition}. We know from the analysis in Section \ref{stab:fol} that for each $\rho$ large enough the corresponding hemisphere is
			a strictly stable free boundary CMC surface graphically described by a function $\phi_\rho$  on $S^2_{\rho,+}(\mathcal C^+)$ satisfying the bound 
			\[
			\|\rho^{-1/2}\phi_\rho\|_{C^{2,\alpha}}^{(\rho)}\leq C,
			\] 
			where $C>0$ is an absolute constant and the weighted H\"older norm is defined as in the left-hand side of (\ref{schauder}). The uniqueness claim is that, for $\rho$ large enough depending only on $C$, any other free boundary CMC hemisphere with the {\em same} mean curvature and which is graphed by a function satisfying this H\"older bound should coincide with (the graph of) $\phi_0:=\phi_\rho$.  Indeed, assume  there exists another such hemisphere, say associated to a function $\phi_1$.  As in \cite[Proposition 2.1]{huisken1996definition}, the asymptotic roundness of the graphs means that we may interpolate between the corresponding embeddings by setting
			\[
			F_{t}(x)=F_{0}(x)+tu(x)\nu(x), \quad t\in[0,1],
			\]     
			for some function $u(x)=\langle\vec{\mathfrak a},\nu(x)\rangle +q(x)$, where $\vec{\mathfrak a}\in\mathbb R^2$ is a vector and  $q=O(\rho^{-1})$. 
			A crucial remark at this point is that all of these surfaces are free boundary (with a possibly non-constant mean curvature $H_{F_t}$ for $0<t<1$) and may be graphed by using functions satisfying the {same} H\"older bound as $\phi_0$.
			Since $H_{F_0}=H_{F_1}$,  the variational vector field $Y=F_1-F_0$ satisfies
			\[
			|Y|\leq\|dH_{F_0}\|^{-1}\sup_t\|d^2H_{F_t}(Y,Y)\| \leq C_1|Y|^2,
			\]
			where we used (\ref{est:invL}) applied to $dH_{F_0}=\mathscr L_{F_0}$, the Jacobi operator associated to $F_0$, and the fact that $\|d^2H_{F_t}\|=O(\rho^{-3})$ uniformly in $t$.  Thus, there exists an absolute constant $C_2>0$ such that $|Y|\leq C_2$ implies $Y=0$. We next check that $|Y|$ (equivalently, $|\vec{\mathfrak a}|$) may be chosen small enough so as to fulfill this vanishing criterion if $\rho$ is large.  We first note that, again because $H_{F_0}=H_{F_1}$, 
			\begin{equation}\label{first:sec}
				\|dH_{F_0}Y\|\leq \sup_{t}\|(dH_{F_t}-dH_{F_0})Y\|.
			\end{equation}
			As in \cite[Proposition 16]{ambrozio2015rigidity} we compute that  
			\[
			dH_{F_t}Y=\mathscr L_{F_t}u+Y^{\top}H_{F_t},
			\]
			where $Y^\top$ is the tangential component of $Y$. 
			Starting with (\ref{normal:unit}) we obtain $|Y^{\top}|=O(\rho^{-3})$ and hence $|Y^{\top}H_{F_t}|=O(\rho^{-4})$. Combining this with  (\ref{est:asym:gr:fb}) we see that the right-hand side of (\ref{first:sec}) is $O(\rho^{-4})$. On the other hand, since $\langle\vec{\mathfrak a},\nu\rangle$ is an approximate eigenfunction of $\mathscr L_{F_0}$ under Neumann boundary condition with eigenvalue close to $6m\rho^{-3}$, the left-hand side of (\ref{first:sec}) is $\geq C_3|\vec{\mathfrak a}|\rho^{-3}$.  Thus, $|\vec{\mathfrak a}|\leq C_4\rho^{-1}$ and the uniqueness claim follows provided we take $\rho\geq C_2^{-1}C_4$.  
			
		\bibliographystyle{plain}
		\bibliography{bibfile-centerofmass}

	\end{document}